\documentclass[12pt]{amsart}
\usepackage{amsfonts, amsmath, amssymb,amsthm}
\usepackage{amsfonts}
\usepackage{enumitem}
\usepackage[a4paper]{geometry}
\usepackage{mathtools}

\usepackage{algorithm}
\usepackage{algorithmic}

\usepackage{color}

\usepackage{hyperref}

\newtheorem{theorem}{Theorem}[section]
\newtheorem{corollary}[theorem]{Corollary}

\newtheorem{lemma}[theorem]{Lemma}
\newtheorem{proposition}[theorem]{Proposition}

\theoremstyle{definition}
\newtheorem{example}[theorem]{Example}
\newtheorem{examples}[theorem]{Examples}
\newtheorem{definition}[theorem]{Definition}

\newtheorem{remark}[theorem]{Remark}
\newtheorem{remarks}[theorem]{Remarks}

\newcommand{\half}{\frac{1}{2}}

\newcommand{\TT}{\mathbb{T}}

\newcommand{\RR}{\mathbb{R}}
\newcommand{\ZZ}{\mathbb{Z}}
\newcommand{\Projn}{\Ran / \RR \One}
\newcommand{\Rn}{\RR^{n}}
\newcommand{\Zn}{\ZZ^{n}}
\newcommand{\Ran}{\RR^{n+1}}

\newcommand{\ta}{\oplus}
\newcommand{\tm}{\odot}

\newcommand{\varx}{{\bf x}}
\newcommand{\vary}{{\bf y}}

\newcommand{\varv}{{\bf v}}

\newcommand{\varp}{{\bf p}}

\newcommand{\One}{{\bf 1}}
\newcommand{\Zero}{{\bf 0}}

\newcommand{\tnorm}[1]{\norm{#1}_{\mathrm{tr}} \,}
\newcommand{\dualtnorm}[1]{\norm{#1}^*_{\mathrm{tr}} \,}
\newcommand{\Bn}{B_{\mathrm{tr}}^n}
\newcommand{\dualBn}{B^{n*}_{\mathrm{tr}}}
\newcommand{\BRn}{B_{\mathrm{tr},R}^n}
\newcommand{\BRtwo}{B_{\mathrm{tr},R}^2}
\newcommand{\Ben}{B_{\mathrm{tr},\varepsilon}^n}

\newcommand{\Bthree}{B_{\mathrm{tr}}^3}

\newcommand{\Bdn}{B_{\mathrm{tr},\delta}^n}

\newcommand{\SRan}{S_{\mathrm{tr},R}^{n-1}}

\newcommand{\ltr}[1]{l_{\mathrm{tr}}(#1)}

\newcommand{\In}[1]{I_{#1}^n}

\newcommand{\Idn}[1]{I_{#1,\delta}^n}
\newcommand{\vol}[1]{{\mathcal{V}(#1)}}
\newcommand{\area}[1]{{\mathcal{A}(#1)}}
\newcommand{\tvol}[1]{{\mathcal{V}_{\mathrm{tr}}(#1)}}
\newcommand{\tarea}[1]{{\mathcal{A}_{\mathrm{tr}}(#1)}}
\newcommand{\art}{\mathcal{A}_{\mathrm{tr}}}
\newcommand{\lent}{l_{\mathrm{tr}}}

\newcommand{\norm}[1]{\parallel #1 \parallel}
\newcommand{\abs}[1]{\lvert #1 \rvert}

\numberwithin{equation}{section}

\usepackage{pgfplots}
\pgfplotsset{compat = newest}
\usepgfplotslibrary{fillbetween}

\usepackage{calc}
\usepackage{tikz}
\usetikzlibrary{decorations.markings}
\tikzstyle{vertex}=[circle, draw, inner sep=0pt, minimum size=6pt]
\usepackage{graphicx}
\graphicspath{ {./images/} }
\usetikzlibrary{shapes.geometric}
\usetikzlibrary{shapes.geometric}
\usetikzlibrary{arrows}

\begin{document}
\title[]{Tropical measures, anisotropic isoperimetric inequality and honeycomb}
\author{Amnon Rosenmann}
\email[]{rosenmann@math.tugraz.at}
%\affil{
	%Institute of Discrete Mathematics \\
	%Graz University of Technology, Graz, Austria \\
%	rosenmann@math.tugraz.at}
\date{}
%\date{\thanks{} \today}
%\maketitle

\begin{abstract}
We introduce a tropical spherical measure on $\mathbb{R}^n$ that is based on the tropical metric and is an analogue of spherical Hausdorff measure. This measure is translation invariant but, unlike Lebesgue measure, is not invariant under rotations or reflections. It agrees with Lebesgue measure on $n$-dimensional (but not on $k$-dimensional, $k<n$) measurable subsets of $\mathbb{R}^n$, and on rectifiable curves it recovers tropical length.
In dimension $2$ we prove a sharp tropical isoperimetric inequality, with equality precisely for tropical disks, and deduce a tropical honeycomb theorem. We also introduce a tropical analogue of Minkowski content and show that the tropical ball is the associated Wulff shape. This yields an anisotropic type of the tropical isoperimetric problem and consequently a tropical honeycomb theorem in $\mathbb{R}^n$.
Finally, we describe the tropical dual norm and dual ball, compare the tropical spherical and Minkowski surface measures, and prove that they agree in the plane and on polytopes in $\mathbb{R}^n$ whose facets are parallel to facets of the tropical ball or its dual.

\end{abstract}
\subjclass[2020]{14T99, 15A80, 28A75, 49Q05, 52A38, 52A40}
\keywords{tropical measure, tropical balls, tropical isoperimetric inequality, anisotropic isoperimetry problem, tropical honeycomb}
\maketitle
\tableofcontents{}

\section{Introduction}
\label{sec:intro}
The isoperimetric inequality is one of the foundational problems in geometric measure theory, studied in various ambient spaces and under various measures and metrics. Here we look at it in the setting of tropical geometry, in which it is an anisotropic isoperimetry problem.
For this end we introduce tropical measures that are based on the tropical metric that was introduced by Cohen, Gaubert and Quadrat \cite{CGQ04}. 
These measures are applied to rectifiable subsets of $\Rn$ and not just confined to polytopes or piecewise linear spaces.
As a consequence of the tropical isoperimetric inequality we obtain a tropical honeycomb theorem in $\Rn$.

Geometric Measure Theory (GMT) \cite{Fed69, Sim83, Mag12, Mor16} emerged in the second half of the 20th century to analyze geometric variational problems, such as minimal surface.
It extends the smooth framework of differential geometry to accommodate singularities and irregularities to handle  rectifiable sets, currents, and varifolds.
 
A more recent field is tropical geometry.
It may be seen as a degeneration of complex structures into piecewise linear and polyhedral structures through Maslov's log-limit ``dequantization'' (see \cite{KM97}).
Tropical geometry studies polynomials and varieties over the semifield of extended real numbers with $+\infty$ (or $-\infty$) and the idempotent operation of minimum (or maximum) in place of addition and the operation of addition in place of multiplication.
The result is a kind of linear and combinatorial and, in a way, simpler version of algebraic geometry. Applications are, among others, in algebraic geometry, optimization problems, discrete event dynamical systems, neural networks and string theory. For the main source on tropical geometry we refer to the book of Maclagan and Sturmfels \cite{MS15}; other selected sources are \cite{IMS09}, \cite{Alg13}, \cite{MR}, \cite{PS05}, \cite{Gro15},\cite{Jos21}. 

A natural question is whether one can transport GMT-type questions to a tropical setting.
This raises the following problem:
how to define the basic notion of measure or volume in the tropical or idempotent context?
Let us look at some of the approaches that were chosen for that matter. 

Maslov \cite{Mas87} (see also \cite{D-MD98}, \cite{Aki99}, ) defined an idempotent measure $\mu$ that is an extreme measure in the sense that $\mu(A \cup B) = \max(\mu(A), \mu(B))$ (or $\min(\mu(A), \mu(B))$), followed by a natural analogue of Lebesgue integral. These notions lead to the introduction by Maslov of an optimization theory as an analogue of probability and stochastic theory, with applications to areas like optimal control and Hamilton--Jacobi equations.

A notion of tropical volume of a polytope, given as the rows of an $n \times n$ matrix $A$, $n \geq 2$, was introduced by Depersin, Gaubert and Joswig \cite{DGJ17}. It is defined as the gap between the maximum assignment problem (the tropical permanent of $A$) and the second best assignment. The volume is zero, that is, the maximum is achieved at least twice, exactly when the rows (or columns) of A are contained in a tropical hyperplane \cite{R-GST05}. In dimension $n=2$ it equals the tropical distance.
The authors also proved an isodiametric inequality theorem with respect to their notion of volume. 

Another definition of a volume of a tropical polytope (a finitely generated tropical convex set) was introduced by Loho and Schymura \cite{LS20}. Their tropical barycentric volume uses tropical analogues of lattice point counting and Ehrhart polynomial, where their lattice is defined on a logarithmic scale.

We mention another work, that of Lee, Li, Lin and Monod \cite{LLLM22}, which uses the tropical metric in dynamic optimal transport, developing tropical analogues of Wasserstein-$p$ distances between probability spaces. When $p=1$ they get an efficient computation of geodesics on the tropical projective torus.

Our aim here is different: we introduce analogues of classical measures in a tropical setting but on general rectifiable subsets of $\Rn$ and not only on polytopes. The measures are based on the tropical metric that was introduced by Cohen, Gaubert and Quadrat in \cite{CGQ04} (see also \cite{Pue14}), an additive version of Hilbert projective metric.

The major tropical $k$-dimensional measure that we introduce, denoted $\sigma^k_{\mathrm{tr}}$, is defined on the tropical projective torus $\Ran / \RR \One$, identified with $\Rn$.
It is a tropical analogue of the $k$-dimensional Hausdorff spherical measure, where $\sigma^k_{\mathrm{tr}}(S)$ is computed via a limit of coverings of $S \subset \Rn$ by tropical balls of dimension $n$.
This measure coincides with Lebesgue measure on $n$-dimensional (Lebesgue-measurable) subsets of $\Rn$, but differs from it on $k$-dimensional subsets of $\Rn$, for $k<n$.
This follows from the fact that $\sigma^k_{\mathrm{tr}}$, like the tropical metric, is not invariant under Lebesgue isometries of $\Rn$: it is invariant under translations, but not under rotations or reflections.
On a rectifiable curve $\gamma$, $\sigma^1_{\mathrm{tr}}(\gamma)$ coincides with the standard definition through the method of rectification by tropical line segments.

Another measure that we introduce is a tropical analogue of Minkowski measure and Minkowski-Steiner surface area. The surface area with respect to the tropical Hausdorff spherical measure and the tropical Minkowski measure is the same for subsets of the plane and also in $\Rn$ when the surface is composed of facets that are parallel to facets of the tropical ball or its dual.
We provide a direct proof of the tropical isoperimetric inequality in the plane. In $\Rn$ the inequality follows from the theory of anisotropic isoperimetric inequality (see, e.g. \cite{Mag12}, \cite{deros24}).

The anisotropic isoperimetric inequality is a generalization of the classical one. In the latter (standard) balls enclose maximum volume for a fixed surface area. In the anisotropic case we are dealing with a direction-dependent ``surface tension", or norm, so that the ``cost" to enclose volume changes with direction, leading to optimal shapes like convex polytopes (Wulff shapes) instead of balls. Applications are, for example, in determining the equilibrium shape, due to surface energy minimization, of crystals or liquid drops of fixed volume inside a separate phase.

\medskip
\noindent\textbf{Organization.}
Sections 2, 3 and 4 give a background on tropical metric and tropical balls. Section~\ref{sec:proj} contains a short review of tropical algebra and the construction of the tropical projective torus.
In Section~\ref{sec:metric} we recall the definition of tropical metric and then, from \cite{Ros26}, the definitions of tropical length of rectifiable curves and of tropical geodesics. Following, we recall in Section~\ref{sec:ball} some characteristics of tropical balls (also from \cite{Ros26}), like forming a honeycomb tiling of $\Rn$.

In Section~\ref{sec:tropical_measure} we define the tropical spherical measure and show that it equals Lebesgue measure on $n$-dimensional subsets of $\Rn$. We also discuss some alternative ways to define a tropical measure.
Then we show in Section~\ref{sec:surface} that the ratio of surface area to volume (with respect to the tropical spherical measure) in tropical balls is the same as in standard balls.

In Section~\ref{sec:isoperim_R2} we prove the isoperimetric inequality in the plane with respect to the tropical measure, followed by a tropical honeycomb theorem in the plane in 
Section~\ref{sec:honecomb_R2}.	
Section \ref{sec:isoperim_Rn} is about the anisotropic isoperimetric inequality in $\Rn$.
In the first two subsections we give a short description of the isoperimetric and anisotropic isoperimetric inequalities in $\Rn$. Then we describe the tropical dual unit ball which is the Wulff shape with respect to the tropical surface tension.
Next, we show that the tropical ball is the Wulff shape with respect to the tropical Minkowski measure and that the tropical (spherical) measure and the tropical Minkowski measure agree on the area of $\partial{S}$ when $S$ is in the plane or when $S$ is a polytope with facets parallel to those of the tropical ball or its dual, but in general, these measures do not agree with each other.
We conclude in Section~\ref{sec:honecomb_Rn} with a tropical honeycomb theorem in $\Rn$.

\section{The tropical projective torus}
\label{sec:proj}
Tropical algebra is part of ``idempotent mathematics'', which was developed mainly by Maslov and his collaborators (see \cite{Lit07} for a brief introduction).
The (min) {\it tropical semifield} (or {\it min-plus algebra}) ($\TT,\ta, \tm$) consists of the set $\TT = \RR \cup \{\infty\}$ on which the operations of tropical addition $\ta$ and tropical multiplication $\tm$ are defined by
\begin{equation}
	a \ta b := \min(a,b) , \qquad a \tm b := a+b.
\end{equation}
The identity element for addition is $\infty$ and the identity element for multiplication is $0$.
The $n$-dimensional semimodule $\TT^n$ over $\TT$ is then defined with respect to the above operations.
The {\it max-plus algebra} (see the monograph of Butkovi{\v{c}} \cite{But10}) over $\RR \cup \{-\infty\}$ is defined in a similar way, where addition is replaced by the maximum operation and $-\infty$ is the identity element for addition.

Since we are only interested here in bounded objects, we work over $\TT^{\times} = \TT~\setminus~\{\infty\} = \RR$ and the $n$-dimensional {\it tropical projective torus} $\Projn$, where $\One = (1,\ldots,1)$, that is, for all $a \in \RR$, $(x_1,\ldots,x_{n+1}) \sim (x_1+a,\ldots,x_{n+1}+a)$.

For each element of $\Projn$, written in homogeneous coordinates as $(x_1 : \ldots : x_{n+1})$, we can choose as a representative the element $(x_1 - x_{n+1} : \ldots : x_n - x_{n+1}:0)$, and then map it bijectively to the element $(x_1 - x_{n+1}, \ldots, x_n - x_{n+1})$ of $\Rn$. 
In what follows, when we use the coordinates $(x_1,\ldots,x_{n})$ we refer to this representation of $(x_1 : \ldots : x_n : 0) \in \Projn$ as an element of $\Rn$.
For convenience, we also make use of the following coordinate system of $\Rn$.
Let the {\it (min) tropical standard unit vectors} be ${\bf \tilde{e}_1} = (1,0,\ldots,0), \ldots, {\bf \tilde{e}_n} = (0,\ldots,0,1)$ and ${\bf \tilde{e}_{n+1}} = -\One =(-1,\ldots,-1)$. These vectors are defined in \cite{Ros26} to be tropically orthogonal.
Then we define the (min) tropical coordinate system of $\Rn$ with respect to the axes in the $n+1$ directions of these unit vectors.
$\Rn$ is then decomposed into $n+1$ orthants $\Rn_j$, for $j=1,\ldots, n+1$, where $\Rn_j$ is the positive cone determined by the $n$ vectors ${\bf \tilde{e}_1},\ldots,{\bf \widehat{\tilde{e}}_j},\ldots,{\bf \tilde{e}_{n+1}}$.

\section{Tropical metric}
\label{sec:metric}
The tropical metric introduced by Cohen, Gaubert and Quadrat in \cite{CGQ04} is defined as follows.
The {\it tropical distance} in $\Projn$ between $\varx = (x_1 : \ldots : x_{n+1})$ and $\vary = (y_1 : \ldots : y_{n+1})$ is
\begin{equation}
	\label{eq:dist}
	d_{\mathrm{tr}}(\varx, \vary) := \max_{1 \leq i,j \leq n+1} \{x_i - y_i -x_j + y_j\}.
\end{equation}
When mapping $(x_1 : \ldots : x_{n+1}) \in \Projn$ to $(x_1 - x_{n+1}, \ldots, x_n - x_{n+1}) \in \Rn$ as above then the tropical distance in $\Rn$ between $\varx=(x_1,\ldots,x_n)$ and $\vary=(y_1,\ldots,y_n)$ is
\begin{eqnarray}
	d_{\mathrm{tr}}(\varx, \vary) &:=& \max\{\max_{1 \leq i \leq n} \{\abs{x_i - y_i}\}, \max_{1 \leq i,j \leq n} \{x_i - y_i -x_j + y_j\}\} \\
	&=& \max\{\max_{1 \leq i \leq n} \{x_i - y_i\},0\}
	- \min\{ \min_{1 \leq i \leq n} \{ x_i - y_i\}, 0\}. \nonumber
\end{eqnarray}
The {\it tropical norm} of $\varx = (x_1 : \ldots : x_{n+1}) \in \Projn$ is
\begin{equation}
		\tnorm{\varx} := \max_{1 \leq i, j \leq n+1} \{x_i - x_j\}.
\end{equation}
In $\Rn$ the tropical norm of $\varx=(x_1,\ldots,x_n)$ is
\begin{equation}
		\tnorm{\varx} := \max\{\max_{1 \leq i \leq n} \{x_i\}, 0\} - 
		\min \{\min_{1 \leq i \leq n} \{x_i\}, 0\}.
\end{equation}
The {\it tropical line segment} (see \cite{MS15}, p. 230) between points $\varx$ and $\vary$ in $\Rn$ is a piecewise linear curve. 
Each line segment is in direction of a vector with zeros and ones and
the tropical distance between $\varx$ and $\vary$ is the sum of the tropical lengths of these line segments.
In \cite{Ros26} we extended the definition of the tropical length of such piecewise linear curves  to more general curves as follows.
\begin{definition}
	\label{def:length}
	Let $\gamma : [a,b] \to \Rn$ be a finite parametrically-defined rectifiable curve (that is, the mapping $\gamma$ is continuous, Lipschitz and almost everywhere differentiable).
	Let
	\begin{equation}
		l_{\mathrm{tr}, \varepsilon}(\gamma) := \inf_{\mathcal{T}_{\alpha}=(t_{\alpha,i})_{i=0}^{n_\alpha}} \left\{ \sum_{i=1}^{n_{\alpha}} d_{\mathrm{tr}}(\gamma(t_{\alpha,i-1}), \gamma(t_{\alpha,i})) \right\},
	\end{equation}
	where the infimum is over all partitions $\mathcal{T}_{\alpha}=(t_{\alpha,i})_{i=0}^{n_\alpha}$ of $[a,b]$ satisfying $a=t_{\alpha,0} < t_{\alpha,1} < \cdots < t_{\alpha,n_{\alpha}}=b$, with $d_{\mathrm{tr}}(\gamma(t_{\alpha,i-1}), \gamma(t_{\alpha,i})) \leq \varepsilon$, for all $i$. Then the {\it tropical length} of $\gamma$ is
	\begin{equation}
		l_{\mathrm{tr}}(\gamma) := \lim_{\varepsilon \to 0^+} l_{\mathrm{tr},\varepsilon}(\gamma).
	\end{equation}
	Equivalently, the tropical length of $\gamma$ is the total variation:
	\begin{equation}
		\label{eq:total_variation}
		l_{\mathrm{tr}}(\gamma) := \sup_{\mathcal{T}_{\alpha}=(t_{\alpha,i})_{i=0}^{n_\alpha}} \left\{ \sum_{i=1}^{n_{\alpha}} d_{\mathrm{tr}}(\gamma(t_{\alpha,i-1}), \gamma(t_{\alpha,i})) \right\},
	\end{equation}
	where the supremum is over all partitions $\mathcal{T}_{\alpha}=(t_{\alpha,i})_{i=0}^{n_\alpha}$ of $[a,b]$ satisfying $a=t_{\alpha,0} < t_{\alpha,1} < \cdots < t_{\alpha,n_{\alpha}}=b$.
\end{definition}
Then, the standard definition of a tropical geodesic is the following (see \cite{Ros26}).
\begin{definition}
	A curve $\gamma$ between the points $\varx$ and $\vary$ is a {\it tropical geodesic} if $l_{\mathrm{tr}}(\gamma) = d_{\mathrm{tr}}(\varx, \vary)$.
	A set $S \subset \Rn$ is {\it tropically geodesic} if it contains all tropical geodesics between every two points $\varx, \vary \in S$.
\end{definition}

\section{The tropical unit ball}
\label{sec:ball}
The $n$-dimensional tropical (closed) ball of radius $R$ with center at $\varp \in \Rn$ is
	\begin{equation}
		\BRn(\varp) := \{\varx \in {\Rn} : d_{\mathrm{tr}}(\varp, \varx) \leq R \}.
	\end{equation}
We denote by $\BRn$ the tropical ball of radius $R$ that is centered at the origin and when, in addition, the radius is 1 then it is denoted by $\Bn$.
The tropical ball is a convex polytope. When centered at the origin then it is centrally symmetric but not a regular polytope: it is stretched along the directions $\bf 1$ and $-\bf 1$ and shrunk in the orthogonal directions of the hyperplane $x_1 + \cdots + x_n = 0$ (see Figure~\ref{fig:tr_balls}).
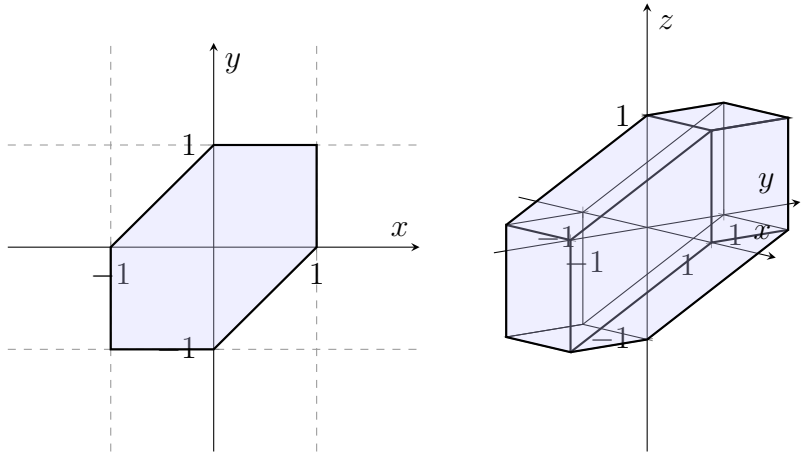
\begin{figure}[h]
	\centering
	\begin{tikzpicture}
		\begin{axis}[
			axis x line=middle,
			axis y line=middle,
			grid = major,
			width=7cm,
			height=7 cm,
			grid style={dashed, gray!80},
			xmin=-2.0,     % start the diagram at this x-coordinate
			xmax= 2.0,    % end   the diagram at this x-coordinate
			ymin= -2.0,     % start the diagram at this y-coordinate
			ymax= 2.0,   % end the diagram at this y-coordinate
			xlabel=$x$,
			ylabel=$y$,
			/pgfplots/xtick={-1.0, 0.0, 1.0}, % make steps of length 0.1
			/pgfplots/ytick={-1.0, 0.0, 1.0}, % make steps of length 0.1
			]
			\draw[thick,black] (1.0,0.0) -- (1.0,1.0) -- (0.0,1.0) -- (-1.0,0.0) -- (-1.0,-1.0) -- (0.0,-1.0) -- cycle [fill=blue!20, opacity=0.3];
			\draw[thick,black] (1.0,0.0) -- (1.0,1.0) -- (0.0,1.0) -- (-1.0,0.0) -- (-1.0,-1.0) -- (0.0,-1.0) -- cycle;
		\end{axis}
	\end{tikzpicture}
	\qquad
	\begin{tikzpicture}
		\begin{axis}[
			view={50}{10},
			axis x line=middle,
			axis y line=middle,
			axis z line=middle,
			grid = major,
			width=9cm,
			height=9cm,
			grid style={dashed, gray !80},
			xmin=-2,     % start the diagram at this x-coordinate
			xmax= 2,    % end the diagram at this x-coordinate
			ymin= -2,     % start the diagram at this y-coordinate
			ymax= 2,   % end the diagram at this y-coordinate
			zmin= -2,     % start the diagram at this z-coordinate
			zmax= 2,   % end the diagram at this z-coordinate
			xlabel=$x$,
			ylabel=$y$,
			zlabel=$z$,
			/pgfplots/xtick={-1.0, 0.0, 1.0}, % make steps of length 0.1
			/pgfplots/ytick={-1.0, 0.0, 1.0}, % make steps of length 0.1
			/pgfplots/ztick={-1.0, 0.0, 1.0}, % make steps of length 0.1
			]
			\coordinate (A1) at (1,0,0);
			\coordinate (A2) at (1,1,0);
			\coordinate (A3) at (0,1,0);
			\coordinate (A4) at (0,0,1);
			\coordinate (A5) at (1,0,1);
			\coordinate (A6) at (1,1,1);
			\coordinate (A7) at (0,1,1);
			
			\coordinate (B1) at (-1,0,0);
			\coordinate (B2) at (-1,-1,0);
			\coordinate (B3) at (0,-1,0);
			\coordinate (B4) at (0,0,-1);
			\coordinate (B5) at (-1,0,-1);
			\coordinate (B6) at (-1,-1,-1);
			\coordinate (B7) at (0,-1,-1);
			
			\draw [thick] (A1) -- (A2) -- (A6) -- (A5) -- cycle;
			\draw (A3) -- (A2) -- (A6) -- (A7) -- cycle;
			\draw [thick] (A4) -- (A5) -- (A6) -- (A7) -- cycle;
			\draw (B1) -- (B2) -- (B6) -- (B5) -- cycle;
			\draw [thick] (B3) -- (B2) -- (B6) -- (B7) -- cycle;
			\draw (B4) -- (B5) -- (B6) -- (B7) -- cycle;
			\draw (A7) -- (B1);
			\draw [thick] (A4) -- (B2);
			\draw (A3) -- (B5);
			\draw [thick] (A5) -- (B3);
			\draw (A2) -- (B4);
			\draw [thick] (A1) -- (B7);
			\draw [thick] (B7) -- (B4) -- (A2);
			\draw [thick] (A2) -- (B4) -- (B7) -- (B6) -- (B2) -- (A4) -- (A7) -- (A6) -- cycle [fill=blue!20, opacity=0.3];
		\end{axis}
	\end{tikzpicture}
	\caption{The tropical unit ball in the plane (a hexagon, left) and in space (a dodecahedron with four-sided faces, right)}
	\label{fig:tr_balls}
\end{figure}

The $n$-dimensional tropical unit ball $\Bn$ is the disjoint union (up to common facets) of the $n+1$ tropical unit hypercubes $\In{j}$, $j = 1, \ldots, n+1$.
The latter are defined as follows (see \cite{Ros26}).
The $n$-dimensional (min) tropical unit hypercube of type $j$, for $j=1,\ldots,n+1$, with base point at the origin, denoted $\In{j}$, is the zonotope defined by the $n$ tropically orthogonal tropical unit vectors
${\bf \tilde{e}_1},\ldots,{\bf \widehat{\tilde{e}}_j},\ldots,{\bf \tilde{e}_{n+1}}$.
When the base point is at the point $\varp$ (a translation of $\In{j}$ by $\varp$)
then it is denoted $\In{j}(\varp)$. 
In Figure~\ref{fig:decomposition} we see the decomposition into tropical hypercubes in the plane.

Another description of $\Bn$ is as the zonotope which is the Minkowski sum of the intervals $[0, {\bf \tilde{e}_i}]$, $i=1,\ldots,n+1$ (see \cite{Ros26}).
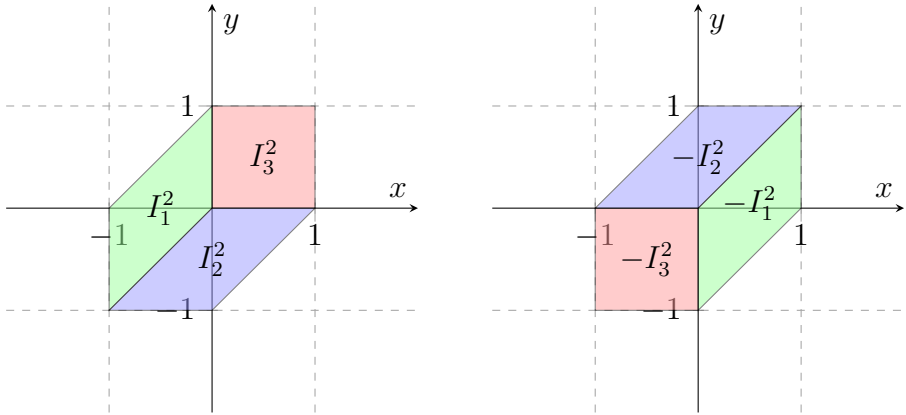
\begin{figure}[h]
	\centering
	\begin{tikzpicture}
		\begin{axis}[
			axis x line=middle,
			axis y line=middle,
			grid = major,
			width=7cm,
			height=7 cm,
			grid style={dashed, gray!70},
			xmin=-2.0,     % start the diagram at this x-coordinate
			xmax= 2.0,    % end   the diagram at this x-coordinate
			ymin= -2.0,     % start the diagram at this y-coordinate
			ymax= 2.0,   % end the diagram at this y-coordinate
			xlabel=$x$,
			ylabel=$y$,
			/pgfplots/xtick={-1.0, 0.0, 1.0}, % make steps of length 0.1
			/pgfplots/ytick={-1.0, 0.0, 1.0}, % make steps of length 0.1
			]
			\draw (0.0,0.0) -- (1.0,0.0) -- (1.0,1.0) -- (0.0,1.0) -- cycle [fill=red!40, 	opacity=0.5];
			\draw (0.0,0.0) -- (0.0,1.0) -- (-1.0,0.0) -- (-1.0,-1.0) -- cycle [fill=green!40, 	opacity=0.5];
			\draw (0.0,0.0) -- (-1.0,-1.0) -- (0.0,-1.0) -- (1.0,0.0) -- cycle [fill=blue!40, 	opacity=0.5];
			\node at (0.5,0.5) {$I^2_3$};
			\node at (-0.5,0.0) {$I^2_1$};
			\node at (0.0,-0.5) {$I^2_2$};
		\end{axis}
	\end{tikzpicture}
	\qquad
	\begin{tikzpicture}
		\begin{axis}[
			axis x line=middle,
			axis y line=middle,
			grid = major,
			width=7cm,
			height=7 cm,
			grid style={dashed, gray!70},
			xmin=-2.0,     % start the diagram at this x-coordinate
			xmax= 2.0,    % end   the diagram at this x-coordinate
			ymin= -2.0,     % start the diagram at this y-coordinate
			ymax= 2.0,   % end the diagram at this y-coordinate
			xlabel=$x$,
			ylabel=$y$,
			/pgfplots/xtick={-1.0, 0.0, 1.0}, % make steps of length 0.1
			/pgfplots/ytick={-1.0, 0.0, 1.0}, % make steps of length 0.1
			]
			\draw (0.0,0.0) -- (-1.0,0.0) -- (-1.0,-1.0) -- (0.0,-1.0) -- cycle [fill=red!40, opacity=0.5];
			%(-1.0,0.0) -- (-1.0,-1.0) -- (0.0,-1.0) -- (1.0,0.0);
			\draw (0.0,0.0) -- (0.0,-1.0) -- (1.0,0.0) -- (1.0,1.0) -- cycle [fill=green!40, opacity=0.5];
			\draw (0.0,0.0) -- (1.0,1.0) -- (0.0,1.0) -- (-1.0,0.0) -- cycle [fill=blue!40, opacity=0.5];
			\node at (-0.5,-0.5) {$-I^2_3$};
			\node at (0.5,0.05) {$-I^2_1$};
			\node at (0.0,0.5) {$-I^2_2$};
		\end{axis}
	\end{tikzpicture}
	\caption{Decomposition of the 2-dimensional tropical unit ball into tropical unit hypercubes: min-plus decomposition (left) and max-plus decomposition (right)}
	\label{fig:decomposition}
\end{figure}

A nice property of tropical balls, which is not shared by regular balls, is that they form a honeycomb tiling of $\Rn$, for every $n \geq 1$ (see Figure~\ref{fig:tessellation} for tiling of the plane). 
\begin{theorem}[see \cite{Ros26}]
	\label{thm:honeycomb}
	The collection of $n$-dimensional tropical unit balls $\Bn({\bf c})$ with centers at the sublattice of $\Zn$
	\begin{equation}
		\Lambda : \{ {\bf c} = (c_1,\ldots,c_n) \in\Zn : \sum_{i=1}^{n} c_i \equiv 0 \: (\mathrm{mod}~n+1)
		\label{eq:center} 
	\end{equation}
	forms a honeycomb of $\Rn$.
\end{theorem}

For more details on the characteristics and structure of tropical balls see \cite{Ros26}.
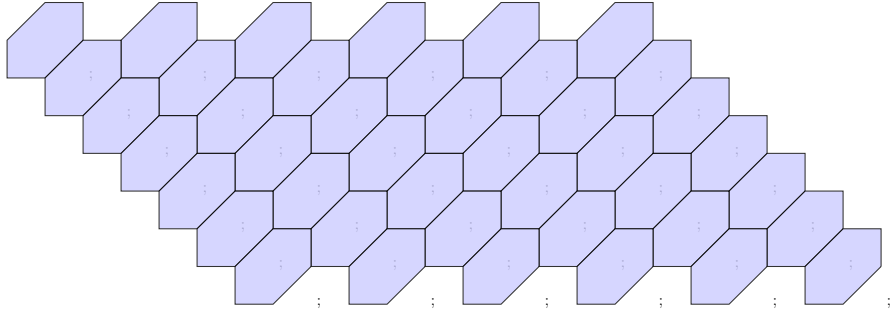
\begin{figure}[h]
	\centering
	\begin{tikzpicture}[scale=0.5, transform shape]
		\foreach \x in {0,...,6} {
			\foreach \y in {0,...,5} {
				\node at (\x+3*\y,-\x)
				{ \tikz\draw [thick,fill=blue!20, opacity=0.8]
					(1.0,0.0) -- (1.0,1.0) -- (0.0,1.0) -- (-1.0,0.0) -- (-1.0,-1.0) -- (0.0,-1.0) -- (1.0,0.0); ;} ;} ;}
	\end{tikzpicture}
	\caption{Tessellation of the plane by tropical disks}
	\label{fig:tessellation}
\end{figure}

\section{Tropical spherical measure}
\label{sec:tropical_measure}
We define a tropical spherical measure (tropical measure, in what follows) $\sigma_{\mathrm{tr}}$ on $\Rn$. It is a translation-invariant metric outer measure that is compatible with the tropical metric.
The tropical measure coincides with Lebesgue measure on $n$-dimensional Lebesgue-measurable subsets of $\Rn$ but they disagree, in general, on subsets of dimension $k < n$ in $\Rn$.
Like the definition of Hausdorff measure (\cite{Fed69}, \cite{Rog70}, \cite{Fal86},\cite{Mor16}), the definition of this tropical measure is by a Carath\'{e}odory's construction.
\begin{definition}
	\label{def:tr_measure}
	Given a set $S$ in $\Rn$, let
\begin{equation}
	\label{eq:tr_epsilon_measure}
	\sigma^k_{\mathrm{tr},\varepsilon}(S) := (k+1) \inf_{\mathcal{C}_{\alpha}} \left\{ \sum_{i=1}^{n_{\alpha}} \delta_{\alpha,i}^k : \delta_{\alpha,i} \leq \varepsilon, \, \bigcup_{i=1}^{n_{\alpha}} B^n_{\mathrm{tr}, \delta_{\alpha,i}}({\bf p_{\alpha,i}}) \supseteq S \right\},
\end{equation}
where the infimum is taken over all countable covers of $S$ by collections 
$$\mathcal{C}_{\alpha} = (B^n_{\mathrm{tr},\delta_{\alpha,i}}({\bf p_{\alpha,i}}))_{i=1}^{n_\alpha}, \quad n_{\alpha} \leq \infty,$$
of $n$-dimensional tropical balls $B^n_{\mathrm{tr}, \delta_{\alpha,i}}({\bf p_{\alpha,i}})$ with centers at ${\bf p_{\alpha,i}} \in S$ (Minkowski cover) and radii $\delta_{\alpha,i} \leq \varepsilon$.
Then we define the {\it $k$-dimensional tropical measure} of $S$ to be
\begin{equation}
	\label{eq:tr_measure}
	\sigma^k_{\mathrm{tr}}(S) := \lim_{\varepsilon \to 0^+} \sigma^k_{\mathrm{tr}, \varepsilon}(S).
\end{equation}
\end{definition}
We also define
\begin{equation}
	\label{eq:tr_=epsilon_measure}
	\sigma^k_{\mathrm{tr},=\varepsilon}(S) := (k+1) \varepsilon^k \inf_{\mathcal{C}_{\alpha}} \{\mathrm{card}(\mathcal{B}_{\alpha})\},
\end{equation}
where the infimum is taken over all (Minkowski) covers $\mathcal{C}_{\alpha}$ of $S$ by countably many $n$-dimensional tropical balls with centers on $S$ and of radius exactly $\varepsilon$.

\begin{remarks}
	\begin{enumerate}
		\item The term $k+1$ in \eqref{eq:tr_epsilon_measure} and in \eqref{eq:tr_=epsilon_measure} comes from our choice to set the measure of the $k$-dimensional tropical unit hypercube to be 1, thus the measure of the $k$-dimensional tropical unit ball is $k+1$.
		\item The definition of the tropical measure resembles the definition of spherical Hausdorff measure.
		Note, however, that we restrict ourselves to coverings with tropical closed balls
		of dimension $n$ when the ambient space is $\Rn$. This follows from the fact that 
		tropical balls are ``twisted'' in the Euclidean metric and
		in general we cannot use tropical balls of dimension $k<n$ to measure a $k$-dimensional set $S \subset \Rn$ even when $S$ is contained in a $k$-dimensional affine subspace of $\Rn$.
		\item Another point is that the centers of the covering tropical balls are on the covered set $S$, which means that the cover is done by a union of Minkowski sums. We can restrict ourselves to such covers because of the convexity and symmetry properties of the tropical balls.
		\item Suppose that $S$ is a $k$-dimensional set that lies in $\Rn$, $n \geq k$. Then $\sigma^k_{\mathrm{tr}}(S)$ is invariant to the natural embedding of $S$ in $\RR^m$, for $m > n$, since the intersection with $\Rn$ of an $m$-dimensional tropical ball of radius $\delta$ is an $n$-dimensional tropical ball of the same radius $\delta$. 
	\end{enumerate}
\end{remarks}
\begin{example}
	When $S$ is an $n$-dimensional box in $\Rn$ with edges parallel to the tropical coordinate axes and of tropical lengths $a_1,\ldots,a_n$ then the tropical measure of $S$ equals its Euclidean volume, that is, $\sigma^n_{\mathrm{tr}}(S) = a_1 \cdots a_n$.
\end{example}
Next, we show that this equality with Lebesgue measure holds in general.
\begin{theorem}
	\label{thm:n-measure}
	Let $S$ be a Lebesgue-measurable $n$-dimensional subset of $\Rn$. Then
	\begin{equation}
		\sigma^n_{\mathrm{tr}}(S) = \mathcal{L}^n(S),
	\end{equation}
	where $\sigma^n_{\mathrm{tr}}(S), \mathcal{L}^n(S)$ are the $n$-dimensional tropical, respectively, Lebesgue measure of $S$. In particular, the (tropical) volume of a tropical ball of radius $R$ is
	\begin{equation}
	\label{eq:n-measure}
	\sigma^n_{\mathrm{tr}}(\BRn) =(n+1)R^n.
\end{equation}	
\end{theorem}
\begin{proof}
	For ease of notation, in what follows we write $\Bdn$ instead of $\Bdn(\varp)$ for a tropical ball of radius $\delta$ and center $\varp$. By \eqref{eq:tr_measure}, the tropical measure $\sigma^n_{\mathrm{tr}}(S)$ of $S$ is defined through the covering of $S$ by tropical balls $\Bdn$. 
	Each such tropical ball $\Bdn$ of radius $\delta$ is composed of $n+1$ tropical hypercubes $\Idn{j}$ of tropical side length $\delta$. Each $\Idn{j}$ is defined by $n$ tropically orthogonal vectors $\delta{\bf \tilde{e}_j}$ and is of standard volume $\delta^n$ (either $\Idn{j}$ is a standard cube or a parallelotope defined by $n-1$ vectors in directions of the Euclidean coordinate axes and one vector in direction $-\One$, whose orthogonal projection on the remaining Euclidean coordinate axis is of size $\delta$), so that $\mathcal{L}^n(\Bdn) = (n+1)\delta^n$.
	This is also the contribution of $\Bdn$ to the tropical measure of $S$ in \eqref{eq:tr_epsilon_measure},
	and since the tropical balls cover $S$
	we have $\mathcal{L}^n(S) \leq \sigma^n_{\mathrm{tr}}(S)$.
	But by Theorem~\ref{thm:honeycomb}, for every $\varepsilon >0$, there exists $\delta>0$, such that $S$ can be covered by tropical balls of radius at most $\delta$ and such that $\sigma^n_{\mathrm{tr},\delta}(S) < \mathcal{L}^n(S) + \varepsilon$. 
	It follows that $\sigma^n_{\mathrm{tr}}(S) \leq \mathcal{L}^n(S)$ and by the two inequalities we have $\sigma^n_{\mathrm{tr}}(S) = \mathcal{L}^n(S)$.
\end{proof}

\subsection{Tropical length of a curve}
As in the Euclidean setting, where the Hausdorff measure of a curve equals its length, this is also the case in the tropical setting.
\begin{proposition}
	Let $\gamma : [a,b] \to \Rn$ be a rectifiable curve. Then $\sigma^1_{\mathrm{tr}}(\gamma) =  l_{\mathrm{tr}}(\gamma)$.
\end{proposition}	
\begin{proof}
	The tropical length $l_{\mathrm{tr}}(\gamma)$ of $\gamma$ is given in Definition~\ref{def:length}. It uses a series of piecewise linear approximations through the method of rectification, reaching a limit from below.
	The tropical measure $\sigma^1_{\mathrm{tr}}(\gamma)$ of $\gamma$ is given in Definition~\ref{def:tr_measure}.
	When applying the tropical measure, we are covering $\gamma$ with balls $\Bdn(\varp)$, $\delta < \varepsilon$, where $\varp \in \gamma$. Since $\gamma$ is rectifiable, we may choose $\varepsilon$ to be small enough, so that the intersection of $\gamma$ with $\partial \Bdn$ (except maybe at the endpoints of $\gamma$) is exactly at two points $\bf q_1$ and $\bf q_2$. Then the two line segments from $\varp$ to $\bf q_i$, $i=1,2$, are of tropical length $\delta$ each. Setting the points $\varp, \bf q_1, p_2$ to belong to a partition of $\gamma$, we form in this way a piecewise linear approximation of $\gamma$ of total length at most that of the one we get by covering with tropical balls, since the covering balls may have overlaps. It follows that when $\varepsilon \to 0$ we get that $l_{\mathrm{tr}}(\gamma) \leq \sigma^1_{\mathrm{tr}}(\gamma)$.
	
	On the other hand, for any cover of $\gamma$ by tropical balls with a bound $\varepsilon$ on the radius of the balls, we can take the points of intersections of the balls with $\gamma$ as partition points.
	When $\varepsilon$ is small enough the total length $\varepsilon'$ of the overlap of the balls is also small when seeking the infimum in \eqref{eq:tr_epsilon_measure} (this is possible since $\gamma$ is rectifiable). So, when $\varepsilon \to 0$ then $\varepsilon' \to 0$ and by \eqref{eq:total_variation} we have $\sigma^1_{\mathrm{tr}}(\gamma) \leq  l_{\mathrm{tr}}(\gamma)$.
	
	By both inequalities we have $\sigma^1_{\mathrm{tr}}(\gamma) =  l_{\mathrm{tr}}(\gamma)$.
\end{proof}

\subsection{Alternative definitions of a tropical measure}
\label{subsec:alternative_measures}
We present here three alternative definitions of measures in the tropical setting: tropical
Hausdorff measure $\mathcal{H}_{\mathrm{tr}}$, tropical cubical measure $\kappa_{\mathrm{tr}}$ and tropical triangular measure $\tau_{\mathrm{tr}}$.

\subsubsection{Tropical Hausdorff measure}
The tropical Hausdorff measure is the analogue of the classical Hausdorff measure.
\begin{definition}
	Given a non-empty set $S \in \Rn$, its {\it tropical diameter}, $\mathrm{diam}_{\mathrm{tr}}(S)$, is
	\begin{equation}
		\mathrm{diam}_{\mathrm{tr}}(S) := \sup_{\varx,\vary \in S} d_{\mathrm{tr}}(\varx,\vary).
	\end{equation}
\end{definition}
\begin{definition}
	Given a set $S$ in $\Rn$, let
	\begin{equation}
		\mathcal{H}^k_{\mathrm{tr},\varepsilon}(S) := (k+1) \inf_{\mathcal{C}_{\alpha}} \left\{ \sum_{i=1}^{n_{\alpha}} \left( \frac{\delta_{\alpha,i}}{2} \right)^k : \delta_{\alpha,i} \leq \varepsilon, \, \bigcup_{i=1}^{n_{\alpha}} S_{\delta_{\alpha,i}} \supseteq S \right\},
	\end{equation}
	where the infimum taken is over all covers of $S$ by collections $\mathcal{C}_{\alpha} = (S_{\delta_{\alpha,i}})_{i=1}^{n_\alpha}$, $n_{\alpha} \leq \infty$, of $k$-dimensional subsets $S_{\delta_{\alpha,i}}$ with $\mathrm{diam}_{\mathrm{tr}}(S_{\delta_{\alpha,i}}) = \delta_{\alpha,i} \leq \varepsilon$.
	Then we define the {\it $k$-dimensional tropical Hausdorff measure} of $S$ to be
	\begin{equation}
		\label{eq:tr_diam_measure}
		\mathcal{H}^k_{\mathrm{tr}}(S) := \lim_{\varepsilon \to 0^+} \mathcal{H}^k_{\mathrm{tr}, \varepsilon}(S).
	\end{equation}
\end{definition}

\subsubsection{Tropical cubical measure}
\begin{definition}
	Given a set $S$ 
	in $\Rn$, let
	\begin{equation}
	\label{eq:tr_cubic_epsilon_measure}
	\kappa^k_{\mathrm{tr},\varepsilon}(S) := \inf_{\mathcal{C}_{\alpha}} \left\{ \sum_{i=1}^{n_{\alpha}} \delta_{\alpha,i}^k : \delta_{\alpha,i} \leq \varepsilon, \, \bigcup_{i=1}^{n_{\alpha}} I^n_{\alpha,i}(\bf p_{\alpha,i}) \supseteq S \right\},
\end{equation}	
	where the infimum is taken over all countable (of finite or infinite cardinality $n_{\alpha}$) covers $\mathcal{C}_{\alpha}$ of $S$ by $n$-dimensional tropical hypercubes $I^n_{\alpha,i}(\bf p_{\alpha,i})$ of tropical side length $\delta_{\alpha,i} \leq \varepsilon$ and base point ${\bf p_{\alpha,i}} \in S$.
	Then we define the {\it $k$-dimensional tropical cubical measure} of $S$ to be
	\begin{equation}
		\label{eq:tr_cubic_measure}
		\kappa^k_{\mathrm{tr}}(S) := \lim_{\varepsilon \to 0^+} \kappa^k_{\mathrm{tr}, \varepsilon}(S).
	\end{equation}
	\begin{remark}
		In the case of a cover by tropical hypercubes it is essential that the base points $\bf p_{\alpha,i}$ of the cubes are on $S$ (Minkowski cover) because the tropical metric is direction-dependent.
	\end{remark}
\end{definition}

\subsubsection{Tropical triangular measure}
The positive hypercube of side length $R$ in $\Rn$, whose set of $2^n$ vertices is $\{(a_1,\ldots,a_n) \,:\, a_i \in \{0,R\}\}$, is decomposed into $n!$ simplices $\Delta^n_{\mathrm{tr},\sigma,R}$, $\sigma \in S_n$, where
\begin{equation}
	\Delta^n_{\mathrm{tr},\sigma,R} = \{(x_1,\ldots,x_n) \,:\, 0 \leq x_{\sigma(1)} \leq x_{\sigma(2)} \leq \cdots \leq x_{\sigma(n)} \leq R\},
\end{equation}
is a regular tropical simplex of dimension $n$ and of tropical side length $R$. We look at these simplices as having a base point at the origin, which means that they can be translated but not rotated.
\begin{definition}
	Given a set $S$ in $\Rn$, let
	\begin{equation}
		\label{eq:tr_triang_epsilon_measure}
		\tau^k_{\mathrm{tr},\varepsilon}(S) := \frac{1}{k!} \inf_{\mathcal{C}_{\alpha}} \left\{ \sum_{i=1}^{n_{\alpha}} \delta_{\alpha,i}^k : \delta_{\alpha,i} \leq \varepsilon, \, \bigcup_{i=1}^{n_{\alpha}} \Delta^n_{\mathrm{tr}, \delta_{\alpha,i}}({\bf p_{\alpha,i}}) \supseteq S \right\},
	\end{equation}
	where the infimum is taken over all (Minkowski) covers of $S$ by collections
	$$
	\mathcal{C}_{\alpha} = (\Delta^n_{\mathrm{tr}, \delta_{\alpha,i}}({\bf p_{\alpha,i}}))_{i=1}^{n_\alpha}, \quad n_{\alpha} \leq \infty,
	$$
	of $n$-dimensional tropical simplices $\Delta^n_{\mathrm{tr}, \delta_{\alpha,i}}$ of side length $\delta_{\alpha,i} \leq \varepsilon$ and base point ${\bf p_{\alpha,i}} \in S$.
	Then we define the {\it $k$-dimensional tropical triangular measure} of $S$ to be
	\begin{equation}
		\label{eq:triang_measure}
		\tau^k_{\mathrm{tr}}(S) := \lim_{\varepsilon \to 0^+} \tau^k_{\mathrm{tr}, \varepsilon}(S).
	\end{equation}
\end{definition}
\section{Surface area to volume ratio of tropical balls}
\label{sec:surface}
After computing the tropical volume of a tropical ball (Theorem~\ref{thm:n-measure}), we compute here its tropical surface area.
\begin{theorem}
	\label{thm:surface_sphere}
	The tropical surface area of the $(n-1)$-dimensional tropical sphere of radius $R$ is
	\begin{equation}
		\label{eq:surface_sphere}
		\tarea{\SRan} = \sigma^{n-1}_{\mathrm{tr}}(\SRan) = n(n+1)R^{n-1}.
	\end{equation}
\end{theorem}
\begin{proof}
	The $(n-1)$-dimensional tropical sphere of radius $R$ is the boundary of the $n$-dimensional tropical ball $\BRn$, where the latter has $2n$ facets on the hyperplanes $x_j=\pm R$ and $n(n-1)$ facets on the hyperplanes $x_j=x_i + R$ (see \cite{Ros26}).
	For convenience, we may assume these facets pass through the origin.
	We start with the facets $F$ on $x_j=0$ (after translation). 
	Each such facet is a standard $(n-1)$-dimensional hypercube of side $R$, so its Euclidean measure is $R^{n-1}$.
	For its tropical measure, we cover it with tropical balls with centers on $F$.
	The cross-section of the tropical ball $\Ben$ of radius $\varepsilon$ by the hyperplane $x_j=0$ is an $(n-1)$-dimensional tropical ball of radius $\varepsilon$, whose Euclidean as well as tropical measure is $n\varepsilon^{n-1}$. It follows that the tropical measure of $F$ is the same as its Euclidean measure, that is $R^{n-1}$.
	
	After translation, a facet $F$ of the second type
	is of the form $-R \leq x_i \leq 0$, $x_j=x_i$ and $x_i \leq x_k \leq x_i + R$, for $k \neq i,j$. The Euclidean measure of $F$ is $\sqrt{2}R^{n-1}$.
	We can cover $F$ with tropical balls of radius $\varepsilon$ with centers at ${\bf c} = (c_1,\ldots,c_n) \in (\varepsilon \ZZ)^n$, $\sum_{k \neq j} c_k \equiv 0 \: (\mathrm{mod}~n\varepsilon)$, $c_j=c_i$, so that the intersection of two adjacent balls is only at their facets.
	This is possible by Theorem~\ref{thm:honeycomb}, which provides a covering of the orthogonal projection of $F$ on the ($n-1$)-dimensional hyperplane $X_1 \cdots \widehat{X_j} \cdots X_n$, and the projection of the covering in direction $\pm {\bf e_j}$ to the hyperplane $x_i=x_j$, with the centers of the $n$-dimensional tropical balls $\Ben$ on $x_i=x_j$.
	Let $S$ be the cross-section of $\Ben$ by the hyperplane $x_i=x_j$. We have $\sigma^{n-1}_{\mathrm{tr},\varepsilon}(S)=n\varepsilon^{n-1}$, and we want to compute the Euclidean measure of $S$. The measure of the intersection of $S$ with each of the $n+1$ tropical hypercubes that $\Ben$ is composed of is $\mathcal{L}^{n-1}(S \cap I^n_{k,\varepsilon}) = \sqrt{2}\varepsilon^{n-1}$.
	However, $S \cap I^n_{i,\varepsilon} = S \cap I^n_{j,\varepsilon}$ is the common facet of $I^n_{i,\varepsilon}$ and $I^n_{j,\varepsilon}$. So, overall we have $\mathcal{L}^{n-1}(S) = \sqrt{2}n\varepsilon^{n-1}$. We get that the ratio $\sigma^{n-1}_{\mathrm{tr},\varepsilon}(S) : \mathcal{L}^{n-1}(S) = 1 :\sqrt{2}$. Since, by letting $\varepsilon \to 0^+$, $F$ can almost be covered by tropical balls that only intersect in their facets, it follows that $\sigma^{n-1}_{\mathrm{tr}}(F) = \frac{\sqrt{2}R^{n-1}}{\sqrt{2}} = R^{n-1}$.
	
	Since $\SRan$ has $n(n+1)$ facets, each one of area $R^{n-1}$, we have $\tarea{\SRan} = \sigma^{n-1}_{\mathrm{tr}}(\SRan) = n(n+1)R^{n-1}$.
\end{proof}
\begin{corollary}
	\label{cor:ratio}
	The ratio of the tropical surface area to the tropical volume of an $n$-dimensional tropical ball of radius $R$ is
	\begin{equation}
		\tarea{\SRan} : \tvol{\BRn} = n(n+1)R^{n-1} : (n+1)R^n = n : R.
	\end{equation}
\end{corollary}
This is the same ratio that is obtained for balls in Euclidean space:
\begin{equation}
	\area{S^{n-1}_R} : \vol{B^n_R} = 
	\frac{n \pi^{n/2}R^{n-1}}{\Gamma(1+n/2)} :
	\frac{\pi^{n/2}R^n}{\Gamma(1+n/2)} = 
	n : R.
\end{equation}
\begin{remark}
	The above ratio of the surface area to volume with respect to the tropical measure follows also from the fact that the tropical ball is a Wulff shape (see Subsection~\ref{subsec:Wulff_shape}) and Equation \eqref{eq:energy_volume_ratio}, as well as Proposition~\ref{thm:paralle l_polytope}.
\end{remark}
\section{Tropical isoperimetric inequality in the plane}
\label{sec:isoperim_R2}
The famous isoperimetric inequality in the plane (see, e.g., \cite{Oss78}) is the following.
\begin{theorem}[Isoperimetric inequality in $\RR^2$]
	Let $\gamma$ be a simple closed curve in the plane, that is, a continuous injective mapping $\gamma : S^1 \to \RR^2$. Let $l$ be the length of $\gamma$ and $\mathcal{A}$ the area of the enclosed region. Then
	\begin{equation}
		\label{eq:isoperR2}
		l^2 \geq 4\pi \mathcal{A},
	\end{equation}
	with equality holding if and only if $\gamma$ is a circle.
\end{theorem}
Here we show that a similar inequality holds in the tropical setting, where the constant $4\pi$ of \eqref{eq:isoperR2} is replaced by the constant $4 \cdot 3 = 12$.
\begin{lemma}
	\label{lem:polygon iso ineq}
	A Euclidean convex tropical polygon (a two dimensional tropically geodesic compact set) $D  \subset \RR^2$ with tropical perimeter $\lent$ and area $\art$ satisfies:
	\begin{equation}
		\label{eq:tisoper}
		\lent^2 \geq 12 \art,
	\end{equation}
	with equality holding if and only if $D$ is a tropical ball (tropical disk).
\end{lemma}
\begin{proof}
	In \cite{Ros26} it is shown that there are 18 different combinatorial types of compact tropically geodesic two-dimensional sets in the plane. They are of the form of polygons with 3-6 edges. A general set of this form with $6$ edges is seen in Figure~\ref{fig:convex_polygon}, where the lengths $a,b,c$ and $d$ are with respect to the tropical metric. We treat here all 18 types as having $6$ edges by noting that some of the edges of the polygon in Figure~\ref{fig:convex_polygon} are of length $0$.
	As can be seen, the tropical perimeter of each of the 18 types of polygons is
	\begin{eqnarray}
		\lent &=& (a+b)+c+(a+b-d)+(b+c) \nonumber \\
		&=& 2a+3b+2c-d,
	\end{eqnarray}
	and its (tropical) area is
	\begin{eqnarray}
		\art &=& (a+b)(b+c)-\frac{b^2}{2}-\frac{d^2}{2} \nonumber \\
		&=& \frac{b^2}{2}-\frac{d^2}{2}+ab+ac+bc.
	\end{eqnarray}
	Then
	\begin{eqnarray}
		\lent^2&=&(2a+3b+2c-d)^2  \nonumber\\
		&=&4a^2+9b^2+4c^2+d^2+12ab+8ac-4ad+12bc-6bd-4cd
	\end{eqnarray}
	and 
	\begin{eqnarray}
		\lent^2´-12\art&=& 4a^2+3b^2+4c^2+7d^2-4ac-4ad+12bc-6bd-4cd \nonumber \\
		&=& 3(b-d)^2+2(a-c)^2+2(c-d)^2+2(d-a)^2.
	\end{eqnarray}
	It follows that $\lent^2´-12\art \geq 0$, with equality if and only if $a=b=c=d$, which is only obtained when $D$ has $6$ edges of equal tropical length, that is, $D$ is a tropical ball.
\end{proof}

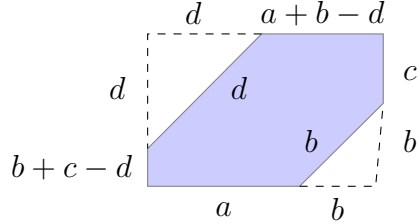
\begin{figure}[h]
	\centering
	\begin{tikzpicture}
		\draw (0,0);
		\begin{scope}
			\draw (0,0) -- (2,0) -- (3.1,1.1) -- (3.1,2.02) -- (1.5,2.02) -- (0,0.5) -- cycle [fill=blue!40, opacity=0.5];
			\draw [dashed] (2,0) -- (3,0) -- (3.1,1.1);
			\draw [dashed] (1.5,2.02) -- (0,2.02) -- (0,0.5);
			\node at (1.0,-0.3){$a$};
			\node at (2.5,-0.3){$b$};
			\node at (3.45,0.6){$b$};
			\node at (2.15,0.6){$b$};
			\node at (3.45,1.5){$c$};
			\node at (2.3,2.3){$a+b-d$};
			\node at (0.6,2.3){$d$};
			\node at (1.2,1.3){$d$};
			\node at (-0.4,1.3){$d$};		
			\node at (-1,0.25){$b+c-d$};
		\end{scope}
	\end{tikzpicture}
	\caption{Euclidean convex tropical polygon}
	\label{fig:convex_polygon}
\end{figure}

Now, we come to the general case. We restrict ourselves to regions $D$ in the plane with a boundary $\partial D$ that is a simple closed curve of finite length (if the length is infinite then the isoperimetric inequality trivially holds).
\begin{theorem}[{\it Tropical isoperimetric inequality in $\RR^2$}]
	\label{thm:R2 iso ineq}
	Let $D$ be a simply connected region in the plane enclosed by a simple closed curve.
	Let $\lent=\ltr{\partial D}$ be the tropical perimeter of $D$ and let $\art=\tarea{D}$ be the (tropical) area of $D$. Then the following tropical isoperimetric inequality holds:
	\begin{equation}
		%		\ltr{\partial D}^2 \geq 12 \tarea{D},
		\lent^2 \geq 12 \art,
	\end{equation}
	with equality holding if and only if $D$ is a tropical ball.
\end{theorem}
\begin{proof}
	We inscribe $D$ in a tropically geodesic polygon $P$, which is of the form $(a_1 \leq x \leq b_1) \wedge (a_2 \leq y \leq b_2) \wedge (a_3 \leq y-x \leq b_3)$, whose edges lie on supporting lines of $D$ (see Figure~\ref{fig:enclosing_polygon}).
	We treat $P$ as having $6$ edges, say $s_0, \ldots, s_5$, although some of the edges may be reduced to a single point.
	Let $\partial P$ be the boundary of $P$ and let it be positively oriented. Let each oriented edge $s_i$ of $P$ be given by a linear map $s_i : [0,1] \to \RR^2$ and let ${\bf q_i}$ be defined by
	$$
	{\bf q_i} = s_i (t_i),
	$$
	where
	$$
	t_i = \inf \{t \in [0,1] : s_i (t) \in \partial P \cap \partial D \}
	$$
	(see Figure~\ref{fig:enclosing_polygon}). The points ${\bf q_i}$ form a division of both $\partial P$ and $\partial D$. For each $i=0,\ldots,5$, let $\partial P_i$ be the part of $\partial P$ from ${\bf q_i}$ to ${\bf q_{i+1 \: (\mbox{mod } 6)}}$, which is a tropical geodesic that is either a point or a (min or max) tropical line segment.
	Hence, its tropical length $\ltr{\partial P_i}$ is smaller or equal to the tropical length $\ltr{\partial D_i}$ of the corresponding sub-curve $\partial D_i$ of $\partial D$ between ${\bf q_i}$ and ${\bf q_{i+1 \, (\mbox{mod } 6)}}$. This implies that
	\begin{equation}
		\label{eq:perim ineq}
		\lent = \sum_{i=0}^{5}\ltr{\partial D_i} \geq \sum_{i=0}^{5}\ltr{\partial P_i} = \ltr{\partial P}.
	\end{equation}
	On the other hand, since $P$ inscribes $D$, we have
	\begin{equation}
		\label{eq:area ineq}
		\art \leq \tarea{P}.
	\end{equation}
	By \eqref{eq:perim ineq} and \eqref{eq:area ineq},
	\begin{equation}
		\label{eq:inequality}
		\frac{\lent^2}{\art} \geq \frac{(l_{\mathrm{tr}}(\partial P)^2}{\tarea{P}},
	\end{equation}
	By \eqref{eq:inequality} and Lemma~\ref{lem:polygon iso ineq}, we have %Theorem~\ref{thm:vol ball} and Theorem~\ref{thm:surface_sphere}, 
	\begin{equation}
		\label{eq:R2 iso ineq}
		\lent^2 \geq \frac{(l_{\mathrm{tr}}(\partial P)^2}{\tarea{P}} \;  \art \geq 
		\frac{(l_{\mathrm{tr}}(S_{\mathrm{tr}}^1))^2 }{\tarea{B_{\mathrm{tr}}^2}} \;  \art = 12 \art,
	\end{equation}
	with equality holding if and only if $D$ is a tropical ball.
\end{proof}
\begin{remark}
	The tropical isoperimetric inequality in the plane also follows from Theorem~\ref{thm:tr_isoperim_inequality} and Proposition~\ref{prop:measures_equality}.
	%	Since in the plane $\lambda^{1}_{\mathrm{tr}}(\partial S) = \sigma^{1}_{\mathrm{tr}}(\partial S)$ then  infers Theorem~\ref{thm:R2 iso ineq}.
\end{remark}
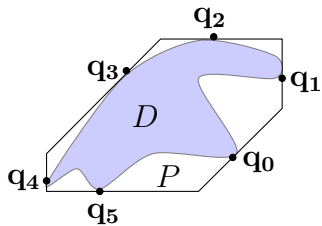
\begin{figure}[h]
	\centering
	\begin{tikzpicture}
		\draw (0,0);
		\begin{scope}
			\draw (0,0) -- (2,0) -- (3.1,1.1) -- (3.1,2.02) -- (1.5,2.02) -- (0,0.5) -- cycle;
			\draw plot [smooth cycle] coordinates {(0,0.1) (1,1.5) (2,2) (3,1.8) (3,1.5) (2,1.5) (2.5,0.5) (1.4,0.5) (0.7,0.03) (0.4,0.3)} [fill=blue!40, opacity=0.5];
			\node at (0.75,-0.3){${\bf q_5}$};
			\node at (2.8,0.39){${\bf q_0}$};
			\node at (3.42,1.45){${\bf q_1}$};
			\node at (2.15,2.3){${\bf q_2}$};
			\node at (0.75,1.6){${\bf q_3}$};
			\node at (-0.3,0.15){${\bf q_4}$};
			\node at (1.3,1){$D$};
			\node at (1.6,0.2){$P$};
			\filldraw [black] (0.7,0) circle (1.2pt);
			\filldraw [black] (2.45,0.45) circle (1.2pt);
			\filldraw [black] (3.1,1.5) circle (1.2pt);
			\filldraw [black] (2.2,2.05) circle (1.2pt);
			\filldraw [black] (1.05,1.6) circle (1.2pt);
			\filldraw [black] (0,0.14) circle (1.2pt);
		\end{scope}
	\end{tikzpicture}
	\caption{A simply connected region $D$ in the plane inscribed in a tropically geodesic set (hexagon) $P$}
	\label{fig:enclosing_polygon}
\end{figure}

\section{Tropical honeycomb theorem in the plane}
\label{sec:honecomb_R2}
The classic honeycomb conjecture states that the most efficient division of the plane into infinitely many regions of the same finite area with a minimal perimeter is into regular hexagons. This ancient problem was proved by Thomas Hales \cite{Hal01}.
A proof for the restricted case of convex polygons was attained 58 years earlier by L\'aszl\'o Fejes  \cite{Fej43}.
\begin{theorem}[Honeycomb theorem in $\RR^2$]
	\label{thm:hexagon}
	Let $G$ be a union of smooth curves in $\RR^2$ that form a locally finite graph and divide the plane into infinitely many bounded simply connected components of unit area. Let $C = \RR^2 \setminus G$ and let $B^2_R \subset \RR^2$ be the disk of radius $R$ centered at the origin. Then
	\begin{equation}
		\label{eq:honecomb}
		\limsup_{R \to \infty}\frac{l(G \cap B^2_R)}{\mathcal{A}(B^2_R)} \geq \sqrt[4]{12} \approx 1.861,
	\end{equation}
	with equality holding for tiling into regular hexagons.
\end{theorem} 
In this inequality, $\mathcal{A}(B^2_R)$ is the area of the disc and $l(G \cap B^2_R)$ is the total length of the part of $G$ at distance at most $R$ from the origin.

We give here the tropical version of this theorem.
\begin{theorem}[Tropical honeycomb theorem in $\RR^2$]
	\label{thm:tr_honecomb}
	Let $G$ be a union of smooth curves in $\RR^2$ that form a locally finite graph and divide the plane into infinitely many bounded simply connected components of (tropical) unit area. Let $C = \RR^2 \setminus G$. Then
	\begin{equation}
		\label{eq:tr_honecomb_R2}
		\limsup_{R \to \infty}\frac{l_{\mathrm{tr}}(G \cap \BRtwo)}{\mathcal{A}_{\mathrm{tr}}(\BRtwo)} \geq \sqrt{3} \approx 1.732,
	\end{equation}
	with equality holding for tiling into tropical regular hexagons (tropical disks).
\end{theorem}
\begin{proof}
	By Theorem~\ref{thm:R2 iso ineq}, the tropical isoperimetric inequality in $\RR^2$, the tropical disk (a tropical regular hexagon), has the minimal perimeter among all regions enclosed by simple closed curves and of the same area.
	In the case of unit area, this minimal perimeter is $\sqrt{\frac{12}{1}} = 2\sqrt{3}$. Moreover, by Theorem~\ref{thm:honeycomb}, these hexagons tile the plane. Since each edge in the tiling belongs to two tropical hexagons, and since the total perimeter of the hexagons that are in the boundary of the tropical disk $\BRtwo$ can be ignored in the ratio when $R \to \infty$ (it is of order $O(R)$, whereas the area of the disk is of order $O(R^2)$), we get that the limit in \eqref{eq:tr_honecomb_R2} is $\sqrt{3}$.  
\end{proof}
\begin{remarks}
	\begin{enumerate}
		\item Note that the lower bound $\sqrt{3}$ of the honeycomb theorem in the plane in the tropical setting is smaller than that in the Euclidean setting, which is $\sqrt[4]{12}$. As for $l^1$-norm (tiling with $l^{\infty}$-discs) and for $l^{\infty}$-norm (tiling with $l^1$-discs) the limit is 2.
		\item Replacing the tropical ball $\BRtwo$ by the standard ball $B^2_R$ in Theorem~\ref{thm:tr_honecomb} does not change the result.
	\end{enumerate}
\end{remarks}

\section{Isoperimetric inequalities in $\Rn$}
\label{sec:isoperim_Rn}
This section is about the anisotropic isoperimetric inequality in different tropical settings. For simplicity and to avoid special cases, we confine ourselves here to sets that we call admissible, implying that definitions and theorems are mostly not stated in their most general form or with respect to open sets, as is common. 
\begin{definition}
	We say that a set $S \in \Rn$ is {\it $k$-admissible} if it is of dimension $k$, compact and has a $(k-1)$-rectifiable boundary of finite perimeter. 
\end{definition}
Also, when a set is a minimizer of an (anisotropic) isoperimetric problem then it is up to translation and dilation and up to a set of measure zero, e.g. when having a ``corona'' of measure zero.

\subsection{Classical isoperimetric inequality}
We denote the $k$-dimensional Lebesgue measure of a set $S \subset \Rn$ by $\mathcal{L}^k(S)$ and its $k$-dimensional Hausdorff measure by $\mathcal{H}^k(S)$.
Hence, $\mathcal{L}^n(B^n)$ is the volume of a (standard) $n$-dimensional unit ball and $\mathcal{L}^n(B^{k}_{\varepsilon})$ is the volume of a $k$-dimensional ball of radius $\varepsilon$.
Recall that a Borel subset $S \subset \Rn$ is $k$-rectifiable if it is of Hausdorff dimension $k$ and there are countably many $C^1$ (or Lipschitz continuous) maps $f_i : \RR^k \to \Rn$ with $\mathcal{H}^k(S \setminus\bigcup_{i=0}^{\infty}f_i(\RR^k))=0$.
The classical isoperimetric inequality in $\Rn$ is the following (see e.g. \cite[Sec. 3.2.43]{Fed69}, \cite{Oss78}, \cite[Thm 14.1]{Mag12}).
\begin{theorem}[Classical isoperimetric inequality in $\Rn$]
	Let $S \subset \Rn$ be an 
	$n$-admissible set.
	Then
	\begin{equation}
		\label{eq:isoper1}
		(\mathcal{H}^{n-1}(\partial S))^n \geq n^n \mathcal{L}^n(B^n) (\mathcal{L}^n(S))^{n-1},
	\end{equation}
	with equality if and only if $S$ is a ball.
\end{theorem}
\subsection{Anisotropic isoperimetric inequality}
The anisotropic isoperimetric inequality is a generalization of the classical isoperimetric inequality, in which the surface area is computed as a weighted integral that is direction dependent (see \cite{Tay78},\cite{Mag12},\cite{Neu20},\cite{FZ22}). An anisotropic character appears, for example, in the surface energy of crystals, where the surface tension is not uniform and the growth of the crystal prefers certain directions that lead to energy minimization. It is then common to use such energy notions for mathematical problems of a similar type.

The {\it anisotropic surface energy} (or $\phi$-surface area) $\Phi(S)$ of an $n$-admissible
set is (see \cite{Mag12})
\begin{equation}
\Phi(S) := \int_{\partial S} \phi(\varv_S(\varx)) \, d\mathcal{H}^{n-1}(\varx),
\end{equation}
where $\varv_S$ is the outer unit normal to $S$ and $\phi : \Rn \to [0,\infty]$, the {\it surface tension} (or $\phi$-norm), is a convex, positive on $S^{n-1}$, one-homogeneous function. In case $\phi(\varv_S)$ is constant (the isotropic case) then $\Phi(S)$ is the (scaled) Euclidean perimeter and the minimizer of the isoperimetric problem is the Euclidean ball. When $\phi(\varv_S)$ is not constant (the anisotropic case) then the minimizer, the {\it Wulff shape} $W_{\phi}$, is a compact convex set given by (see \cite{Mag12}, \cite{Neu20}):
\begin{equation}
	\label{eq:Wulff_shape}
	W_{\phi} = \bigcap_{\varv \in S^{n-1}}	\{ \varx \in \Rn : \langle \varx, \varv \rangle \leq \phi(\varv \},
\end{equation}
and we have (see \cite{Mag12})
\begin{equation}
	\label{eq:energy_volume_ratio}
	\Phi(W_{\phi})=n\mathcal{L}^n(W_{\phi}).
\end{equation}
For example, the Wulff shape for the $l^p$ norm, $p \in [1,\infty]$, is the $l^q$-unit ball, where $l,q$ are Hölder conjugates: $\frac{1}{p}+\frac{1}{q}=1$.

It is not difficult to see that the Wulff shape for the $\phi$-norm is the unit ball with respect to the dual norm $\phi^*$:
\begin{equation}
	\label{eq:Wulff_shape}
	W_{\phi} = \{ \varx \in \Rn : \phi^*(\varx) \leq 1 \},
\end{equation}
where
\begin{equation}
	\label{eq:dual_norm}
	\phi^*(\varx) = \sup_{\varv} \{ \langle \varx, \varv \rangle : \phi(\varv) \leq 1 \}.
\end{equation}

It follows from \eqref{eq:Wulff_shape} that given a compact convex set $S \subset \Rn$ that contains a neighborhood of the origin, it is the Wulff shape for the function $\phi : \Rn \to [0, \infty)$ defined by
\begin{equation}
	\phi(\varx) = \sup \{ \langle \varx, \varv \rangle : \varv \in S \}.
\end{equation}
That is, $\phi(\varx)$ is the {\it support function} of $S$ at $\varx$.

\begin{theorem}[Anisotropic isoperimetric inequality in $\Rn$]
	Let $S \subset \Rn$ be an $n$-admissible set.
	Let $\phi : \Rn \to \RR$ be a positive, convex, coercive, one-homogeneous function and let $W_{\phi}$ be the Wulff shape with respect to the $\phi$-norm. Then
	\begin{equation}
		\label{eq:anisotropic_inequality}
		(\Phi(S))^n \geq n^n \mathcal{L}^n(W_{\phi}) (\mathcal{L}^n(S))^{n-1},
	\end{equation}
	with equality holding if and only if $S=W_{\phi}$.
\end{theorem}

\subsection{The tropical norm as the surface tension and the tropical dual norm}
Since the tropical measure is direction dependent then the isoperimetric inequality with respect to it is an anisotropic isoperimetry problem. 

Let us first see what is the Wulff shape in the case that the surface tension is the tropical norm.
We denote by ${\bf e_0}$ the zero vector $\Zero \in \Rn$ and by $\mathrm{conv}(S)$ the (standard) convex hull of a set $S$.
\begin{proposition}
	 The Wulff shape $W^n_{\mathrm{tr}}=W^n_{\tnorm{\cdot}}$ with respect to the tropical surface tension $\phi_{\mathrm{tr}}(\varx)= \, \tnorm{\varx}$ in $\Rn$ is
	\begin{equation}
		W^n_{\mathrm{tr}} = \mathrm{conv}\{{\bf e_i}-{\bf e_j} : 0 \leq i,j \leq n, \, i\neq j \} \subset \Rn.
	\end{equation}
\end{proposition}
\begin{proof}
	Let $S = \{{\bf e_i}-{\bf e_j} : 0 \leq i,j \leq n, \, i\neq j \}$. These are normal vectors to the facets of the tropical ball. The tropical norm of $\varx \in \Rn$ is
	\begin{equation}
		\tnorm{\varx} = \max_{\varv \in S} \{ \langle \varx, \varv \rangle \}.
	\end{equation}
	Thus, the tropical surface tension is crystalline, for which the unit ball $\dualBn$ with respect to the dual norm
	%=B^n_{\dualtnorm{\cdot}}$,
	is the polytope 
	\begin{equation}
		 \dualBn = \mathrm{conv}(S),
	\end{equation}
	and by \eqref{eq:Wulff_shape}, $W^n_{\mathrm{tr}} = \dualBn = \mathrm{conv}(S)$.
\end{proof}
\begin{example}
	In $\RR^2$ the Wulff shape $W^2_{\mathrm{tr}}$ with respect to the tropical norm, the tropical dual ball, is a rotation by $90^{\circ}$ of the tropical ball $B^2_{\mathrm{tr}}$ (see Figure~\ref{fig:tr_Wulff_shape}).
	In higher dimensions the tropical ball is less similar to its dual: for example, the tropical ball has $2^{n+1}-2$ vertices whereas the tropical dual ball has $n(n+1)$ vertices, which is the number of facets of $B^2_{\mathrm{tr}}$, and similarly for $k$ versus $n-1-k$ faces. 
\end{example}
As is evident from the set of vertices of the tropical unit ball (see \cite{Ros26}) and by \eqref{eq:dual_norm}, the {\it tropical dual norm} of $\varx = (x_1, \ldots,  x_n)$ is
\begin{equation}
	\label{eq:tr_dual_norm}
	\dualtnorm{\varx} = \phi^*_{\mathrm{tr}}(\varx) = \max_{I \subseteq \{1,\ldots,n\}} \{\vert \sum_{i \in I} x_i \vert\},
\end{equation}
that is, the maximum between the sum of the positive entries of $\varx$ and the absolute value of the sum of the negative entries of $\varx$.
Equivalently, it can be computed as follows.
Let
$\tilde{x}_i = \langle \varx, {\bf \tilde{e}_i} \rangle$, for $i=1,\ldots,n+1$, where, as before, ${\bf \tilde{e}_{n+1}} = -\One$.
Then
\begin{equation}
	\dualtnorm{\varx} = \sum_{i=1}^{n+1} \max\{0,\tilde{x}_i\}
	= \frac{1}{2} \sum_{i=1}^{n+1} \vert \tilde{x}_i \vert.
\end{equation}
Note that
\begin{equation}
	\dualtnorm{\varx} \geq \half \sum_{i=1}^{n} \vert x_i \vert = \half \norm{\varx}_1 \, \geq \half \norm{\varx},
\end{equation}
implying that the tropical dual norm is coercive.
On the other hand,
\begin{equation}
	\dualtnorm{\varx} \leq \sum_{i=1}^{n} \vert x_i \vert  = \, \norm{\varx}_1 \,
	\leq \sqrt{n} \norm{\varx}.
\end{equation}
The {\it tropical dual unit ball} is then
\begin{equation}
	\label{eq:tr_dual_unit_ball}
	\dualBn = \{(x_1, \ldots,  x_n) : \max_{I \subseteq \{1,\ldots,n\}} \{\vert \sum_{i \in I} x_i \vert \} \leq 1 \}.
\end{equation}
The {\it tropical dual distance} between $\varx$ and $\vary$ is 
\begin{equation}
	d^*_{\mathrm{tr}}(\varx,\vary) = \, \dualtnorm{\varx-\vary} = \max_{I \subseteq \{1,\ldots,n\}} \{\vert \sum_{i \in I} x_i - y_i \vert\}.
\end{equation}
%\end{remark}
\begin{figure}[h]
	\centering
	\begin{tikzpicture}
		\begin{axis}[
			axis x line=middle,
			axis y line=middle,
			grid = major,
			width=7cm,
			height=7 cm,
			grid style={dashed, gray!80},
			xmin=-2.0,     % start the diagram at this x-coordinate
			xmax= 2.0,    % end   the diagram at this x-coordinate
			ymin= -2.0,     % start the diagram at this y-coordinate
			ymax= 2.0,   % end the diagram at this y-coordinate
			xlabel=$x$,
			ylabel=$y$,
			/pgfplots/xtick={-1.0, 0.0, 1.0}, % make steps of length 0.1
			/pgfplots/ytick={-1.0, 0.0, 1.0}, % make steps of length 0.1
			]
			\draw[thick,black] (1.0,0.0) -- (1.0,1.0) -- (0.0,1.0) -- (-1.0,0.0) -- (-1.0,-1.0) -- (0.0,-1.0) -- cycle [fill=blue!20, opacity=0.3];
			\draw[thick,black] (1.0,0.0) -- (1.0,1.0) -- (0.0,1.0) -- (-1.0,0.0) -- (-1.0,-1.0) -- (0.0,-1.0) -- cycle;
		\end{axis}
	\end{tikzpicture}
	\qquad
	\begin{tikzpicture}
		\begin{axis}[
			axis x line=middle,
			axis y line=middle,
			grid = major,
			width=7cm,
			height=7 cm,
			grid style={dashed, gray!80},
			xmin=-2.0,     % start the diagram at this x-coordinate
			xmax= 2.0,    % end   the diagram at this x-coordinate
			ymin= -2.0,     % start the diagram at this y-coordinate
			ymax= 2.0,   % end the diagram at this y-coordinate
			xlabel=$x$,
			ylabel=$y$,
			/pgfplots/xtick={-1.0, 0.0, 1.0}, % make steps of length 0.1
			/pgfplots/ytick={-1.0, 0.0, 1.0}, % make steps of length 0.1
]
			\draw[thick,black] (1.0,0.0) -- (0.0,1.0) -- (-1.0,1.0) -- (-1.0,0.0) --  (0.0,-1.0) -- (1.0,-1.0) -- cycle [fill=blue!20, opacity=0.3];
			\draw[thick,black] (1.0,0.0) -- (0.0,1.0) -- (-1.0,1.0) -- (-1.0,0.0) --  (0.0,-1.0) -- (1.0,-1.0) -- cycle;
		\end{axis}
	\end{tikzpicture};
	\caption{The tropical unit ball in the plane (left) and its dual, the Wulff shape with respect to the tropical norm (right)}
	\label{fig:tr_Wulff_shape}
\end{figure}
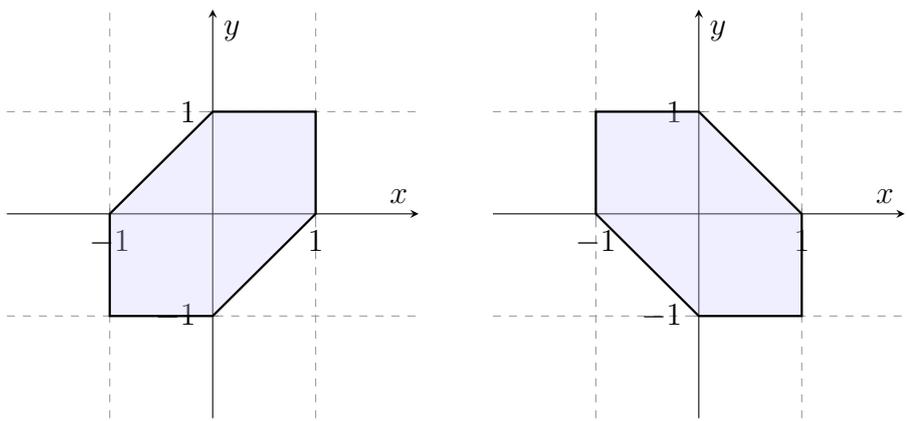

\subsection{The tropical ball as the Wulff shape}
\label{subsec:Wulff_shape}
Let us now examine the case in which the Wulff shape is the tropical unit ball.
By the shape of $\Bn$ and the discussion above we have
\begin{equation}
	\Bn = \mathrm{conv}(S^*),
\end{equation}
where
\begin{equation}
	S^* = \{ \varv^* \in (\{0,1\}^n \cup \{0,-1\}^n) \setminus\{\Zero\}\} \subset \Rn.
\end{equation}
The vectors $\varv^*$ are in the directions of the vertices of $\Bn$ and these are normal vectors to the facets of the dual tropical ball $\dualBn$. The tropical dual surface tension is
\begin{equation}
	\phi^*_{\mathrm{tr}}(\varx)= \, \dualtnorm{\varx} = \max_{\varv^*\in S^*} \{ \langle \varx, \varv^* \rangle \},
\end{equation}
the support function of $S^*$ at $\varx$, which is the support function of $\Bn$ at $\varx$.

The surface norm $\phi^*_{\mathrm{tr}}$ defines a surface measure that can be extended to yet another measure on $\Rn$ that is based on tropical balls, a tropical analogue of Minkowski content, here confined to $k$-admissible sets. 
\begin{definition}
	\label{def:tr_Minkowski_content}
	Let $S \subset \Rn$ be a $k$-admissible set,
	%$k$-rectifiable compact set,
	$k \leq n$. We define the {\it $k$-dimensional tropical Minkowski measure} of $S$ to be
\begin{equation}
	\label{eq:tr_Minkowski_content}
	\mathcal{M}^{k}_{\mathrm{tr}}(S) := 
	%\frac{1}{n-k+1}\lim_{\varepsilon \to 0^+} \frac{\sigma^{n}_{\mathrm{tr}}(S+\Ben)}{\varepsilon^{n-k}}=
	\frac{1}{n-k+1}\lim_{\varepsilon \to 0^+} \frac{\mathcal{L}^{n}(S+\Ben)}{\varepsilon^{n-k}}.
\end{equation}
\end{definition}
\begin{remarks}
	\begin{enumerate}
		\item The term $n-k+1$ in \eqref{eq:tr_Minkowski_content} refers to the tropical volume of a tropical unit ball of dimension $n-k$.
		\item When $k=n$ we have $\mathcal{M}^{n}_{\mathrm{tr}}(S)=\sigma^{n}_{\mathrm{tr}}(S)=\mathcal{L}^{n}_{\mathrm{tr}}(S)$.
		\item For a general set, $\mathcal{M}^{k}_{\mathrm{tr}}(S)$ is not a measure, as is the case with classical Minkowski content.
	\end{enumerate}
\end{remarks}
In particular, for $S \subset \Rn$ an $n$-admissible set, its surface area is
\begin{equation}
	\mathcal{M}^{n-1}_{\mathrm{tr}}(\partial S) := \frac12 \lim_{\varepsilon \to 0^+} \frac{\sigma^{n}_{\mathrm{tr}}(\partial S+\Ben)}{\varepsilon} =
	\frac12 \lim_{\varepsilon \to 0^+} \frac{\mathcal{L}^{n}(\partial S+\Ben)}{\varepsilon}.
\end{equation}
An equivalent derivative-form definition of the surface area is the following.
\begin{definition}
	\label{def:tr_Steiner}
	Let $S \subset \Rn$ be an $n$-admissible set.
	We define the {\it tropical Minkowski-Steiner surface area} of $S$ to be 	\begin{equation}
		\label{eq:tr_Steiner}
		\lambda^{n-1}_{\mathrm{tr}}(\partial S) := %\liminf_{\varepsilon \to 0^+}
		\lim_{\varepsilon \to 0^+} \frac{\mathcal{L}^n(S+\Ben) - \mathcal{L}^n(S)}{\varepsilon}.
	\end{equation}
\end{definition}
As is known, this derivative-type expression defines an anisotropic surface energy $\Phi(S)$ (for a rectifiable surface) because  $h_{\Ben}(\varx)$, the tropical $\varepsilon$-thickening of $S$ at $\varx \in \partial S$ through $S+\Ben$, is exactly $\varepsilon \cdot \phi^*_{\mathrm{tr}}(\varv)$, the (standard) norm of the orthogonal projection of the tropical ball of radius $\varepsilon$ on the unit outer normal $\varv=\varv_S(\varx)$.
Then we have
\begin{equation}
	h_{\Ben}(\varx) = \varepsilon \cdot \dualtnorm{\varv} = \varepsilon \cdot \phi^*_{\mathrm{tr}}(\varv)
	= \varepsilon \cdot \max_{I \subseteq \{1,\ldots,n\}} \{\vert \sum_{i \in I} v_i \vert\}.
\end{equation}
The fact that the tropical Minkowski measure equals Lebesgue measure on $n$-dimensional subsets of $\Rn$ and $\mathcal{M}^{n-1}_{\mathrm{tr}}(\partial S)$ is an anisotropic surface energy that is the support function of the tropical ball is what enables the use of Brunn-Minkowski Inequality in proving that the tropical ball is the Wulff shape. 
The celebrated Brunn-Minkowski Inequality (\cite{Bru1887}, \cite{Min1896}. See also \cite{Gar02}, \cite[Sec. 3.2.41]{Fed69}, \cite{Sch14}) is the following.
\begin{theorem}[Brunn-Minkowski Inequality]
	Let $A,B \in \Rn$ be $n$-dimensional non-empty measurable sets. Then
	\begin{equation}
		(\mathcal{L}^n(A+B))^{1/n} \geq (\mathcal{L}^n(A))^{1/n} + (\mathcal{L}^n(B))^{1/n}.
	\end{equation}  
\end{theorem}
So, for clarity, let us state the theorem about the tropical ball being the Wulff shape with respect to the tropical Minkowski measure. Its proof using Brunn-Minkowski Inequality is similar to the classical proof for the Euclidean ball.
\begin{theorem}[Tropical isoperimetric inequality]
	\label{thm:tr_isoperim_inequality}
	Let $S \subset \Rn$ be an $n$-admissible set.
	Then
	\begin{equation}
		\label{eq:tr_isoper}
		(\mathcal{M}^{n-1}_{\mathrm{tr}}(\partial S))^n \geq 
		n^n (n+1) (\mathcal{M}^{n}_{\mathrm{tr}}(S))^{n-1},
	\end{equation}
	with equality if and only if $S$ is a tropical ball.
\end{theorem}

Next, we compare the tropical measure $\sigma_{\mathrm{tr}}^{n-1}(\partial S)$ with the tropical Minkowski measure $\mathcal{M}_{\mathrm{tr}}^{n-1}(\partial S)$ on the surface of a polytope $S$.
Assume, w.l.o.g., that a facet $F$ of $S$ passes through the origin and is on the hyperplane $\varv^{\perp}$, with $\varv$ a unit vector.
By definition,
\begin{equation}
	\sigma_{\mathrm{tr}}^{n-1}(\varv^{\perp} \cap \Bn) = n.
\end{equation}
It follows that by expressing $\sigma_{\mathrm{tr}}^{n-1}(\partial S)$ as a surface energy then 
\begin{equation}
	\sigma_{\mathrm{tr}}^{n-1}(\partial S) =
	\int_{\partial S} \frac{n}{\mathcal{L}^{n-1}(\varv^{\perp} \cap \Bn)} \, d\mathcal{H}^{n-1}.
\end{equation}
Let $h_{\Bn}(\varv)$ be the support function of $\Bn$ at $\varv$, which is the surface tension $\phi^*_{\mathrm{tr}}(\varv)$. Then
\begin{equation}
	\mathcal{M}_{\mathrm{tr}}^{n-1}(\varv^{\perp} \cap \Bn) = \int_{\varv^{\perp} \cap \Bn} h_{\Bn}(\varv)\, d\mathcal{H}^{n-1} = \mathcal{L}^{n-1}(\varv^{\perp} \cap \Bn) \cdot h_{\Bn}(\varv).
\end{equation}

Extending to $n$-admissible sets, we have the following.
\begin{lemma}
	\label{lem:base_times_height}
	Let $S \in \Rn$ be an $n$-admissible set.
	% let $V$ be the set of all unit normals to the members of $\mathcal{C}$.
	If, for almost all $\varx \in \partial S$,
	\begin{equation}
		\label{eq:cond_base_times_height}
		\mathcal{L}^{n-1}(\varv_S(\varx)^{\perp} \cap \Bn) \cdot h_{\Bn}(\varv_S(\varx))=n,
	\end{equation}
	then $\mathcal{M}_{\mathrm{tr}}^{n-1}(\partial S) = \sigma_{\mathrm{tr}}^{n-1}(\partial S)$.
\end{lemma}
\begin{theorem} 
	\label{thm:paralle l_polytope}
	Let $S \in \Rn$, $n \geq 2$,  be an $n$-dimensional polytope
	%of finite Lebesgue measure and
	with each facet of $S$ parallel to a facet of the tropical ball $\Bn$ or to a facet of the tropical dual ball. 
	Then $\mathcal{M}_{\mathrm{tr}}^{n-1}(\partial S) = \sigma^{n-1}_{\mathrm{tr}}(\partial S)$.
\end{theorem}
\begin{proof}
	Let $F \subset \partial S$ be a facet of $S$ and we may assume, w.l.o.g., that $F$ passes through the origin, so that $F$ lies on a hyperplane $\varv^{\perp}$, $\varv = (v_1,\ldots,v_n) \in \Rn$ a unit vector.
	If $F$ is parallel to a facet of $\Bn$ then either $\varv = \pm {\bf e_i}$, for some $1 \leq i \leq n$, or $\varv = \frac{1}{\sqrt{2}}({\bf e_i}-{\bf e_j})$, for some $1 \leq i \neq j \leq n$. 
	If $F$ is parallel to a facet of the dual ball, then $\varv$ is in direction of a vertex of $\Bn$, so that $\varv = \pm \frac{1}{\sqrt{k}}\sum_{j=1}^{k} {\bf e_{i_j}}$, for some $1 \leq k \leq n$ and $1 \leq i_1 < \cdots < i_k \leq n$.
	So, we may assume, w.l.o.g., that either $\varv = \frac{1}{\sqrt{k}}\sum_{i=1}^{k} {\bf e_{i}}$ or $\varv = \frac{1}{\sqrt{2}}({\bf e_1}-{\bf e_2})$.
	
	As mentioned above, $\Bn$ is the zonotope
	\begin{equation}
		\Bn =  \{ \varx = \sum_{i=1}^{n+1} a_i {\bf \tilde{e}_i} : 0 \leq a_i \leq 1\}.
	\end{equation}
	That is, when $\varx \in \Bn$ then
	\begin{equation}
		\label{eq:ball_element}
		\varx = (a_1-a_{n+1}, \ldots, a_n-a_{n+1}),
	\end{equation}
	where $a_i \in [0,1]$, for $i=1,\ldots,n+1$.
	
	Let $\varv = \frac{1}{\sqrt{k}}\sum_{i=1}^{k} {\bf e_{i}}$ and $\varx \in \varv^{\perp} \cap \Bn$. Since $\langle \varx, \varv \rangle = 0$ then $a_{n+1} = \frac{1}{k}\sum_{i=1}^{k} a_i$ and
	\begin{equation}
		\varx = \sum_{i=1}^{k} a_i({\bf e_{i}} - \frac{1}{k} \One) + \sum_{j=k+1}^{n} a_j{\bf e_{j}}.
	\end{equation}
	It follows that $\varv^{\perp} \cap \Bn$ is the zonotope generated by the vectors
	\begin{equation}
		\vary_i = {\bf e_{i}} - \frac{1}{k} \One, \; \mathrm{for} \; i=1,\ldots,k, \; \mathrm{and} \; {\bf \vary_j = e_{j}}, \; \mathrm{for} \;  j=k+1,\ldots,n.
	\end{equation}
	Since $v$ is a unit vector that is orthogonal to the $(n-1)$-dimensional central section $\varv^{\perp} \cap \Bn$, then we can add it to the generators of $\varv^{\perp} \cap \Bn$ without changing the volume, and we get the following formula.
	\begin{equation}
			\mathcal{L}^{n-1}(\varv^{\perp} \cap \Bn) = \sum_{i=1}^{n} \vert \det (\varv, \vary_1, \ldots, \widehat{\vary}_i, \ldots, \vary_n) \vert = \frac{n}{\sqrt{k}},
	\end{equation}
	since one verifies that each of the $n$ determinants evaluates to $\pm \frac{1}{\sqrt{k}}$.
	\begin{equation}
		h_{\Bn}(\varv) = \max_{I \subseteq \{1,\ldots,n\}} \{\vert \sum_{i \in I} v_i \vert\} = k \cdot \frac{1}{\sqrt{k}} = \sqrt{k}.
	\end{equation}
	It follows that
	\begin{equation}
		\mathcal{L}^{n-1}(\varv^{\perp} \cap \Bn) \cdot h_{\Bn}(\varv) = \frac{n}{\sqrt{k}} \cdot \sqrt{k} = n.
	\end{equation}
	
	The other case is when $\varv = \frac{1}{\sqrt{2}}({\bf e_1}-{\bf e_2})$. As before, when $\varx \in \Bn$ then it can be expressed as in \eqref{eq:ball_element} and the condition $\langle \varx, \varv \rangle = 0$ translates to $a_1=a_2$. Thus,  
	\begin{equation}
		\varx = a_1({\bf \tilde{e}_1} + {\bf \tilde{e}_2}) + \sum_{i=3}^{n+1} a_i {\bf \tilde{e}_i} : 0 \leq a_i \leq 1\}.
	\end{equation}
	Thus, $\varv^{\perp} \cap \Bn$ is the zonotope generated by the vectors
	\begin{equation}
		\vary_1={\bf e_1} + {\bf e_2}, \; \vary_i = {\bf e_{i+1}}, \; \mathrm{for} \; i=2,\ldots,n-1, \; \mathrm{and} \; \vary_n = -\One.
	\end{equation}
	As before, we compute the volume of the central section of $\Bn$ as
	\begin{equation}
		\mathcal{L}^{n-1}(\varv^{\perp} \cap \Bn) = \sum_{i=1}^{n} \vert \det (\varv, \vary_1, \ldots, \widehat{\vary}_i, \ldots, \vary_n) \vert = \sqrt{2}n,
	\end{equation}
	since each of the $n$ determinants evaluates to $\pm \sqrt{2}$.
	\begin{equation}
		h_{\Bn}(\varv) = \max_{I \subseteq \{1,\ldots,n\}} \{\vert \sum_{i \in I} v_i \vert\} = v_1 = \frac{1}{\sqrt{2}}.
	\end{equation}
	We get
	\begin{equation}
		\mathcal{L}^{n-1}(\varv^{\perp} \cap \Bn) \cdot h_{\Bn}(\varv) = \sqrt{2}n \cdot \frac{1}{\sqrt{2}} = n.
	\end{equation}

	Since $\mathcal{L}^{n-1}(\varv^{\perp} \cap \Bn)\cdot h_{\Bn}(\varv) = n$ for all facets of $S$ then by Lemma~\ref{lem:base_times_height},
	$\mathcal{M}_{\mathrm{tr}}^{n-1}(\partial S) = \sigma^{n-1}_{\mathrm{tr}}(\partial S)$.
\end{proof}

We have seen that $\mathcal{M}^{n-1}_{\mathrm{tr}}(\partial S) = \sigma^{n-1}_{\mathrm{tr}}(\partial S)$ in the case that $S$ is a polytope with facets parallel to those of a tropical ball.
This implies that the tropical isoperimetric inequality as stated in Theorem~\ref{thm:tr_isoperim_inequality} holds also for the tropical measure $\sigma_{\mathrm{tr}}$, as long as we confine ourselves to polytopes with facets parallel to those of a tropical ball.
Clearly, equality between the two measures holds in dimension $n=1$.
For $n \geq 2$, we need to check that condition~\eqref{eq:cond_base_times_height}, $\mathcal{L}_{\mathrm{tr}}^{n-1}(\varv^{\perp} \cap \Bn) \cdot h_{\Bn}(\varv)=n$ is satisfied.
This can be easily shown to hold for $n=2$ through elementary trigonometric arguments ($\frac{2}{\cos(\alpha)}\cdot\cos(\alpha)=2$ or $\frac{\sqrt{2}}{\cos(\alpha)}\cdot\sqrt{2}\cos(\alpha)=2$).
By approximations with polygons, we have the following.
\begin{proposition}
	\label{prop:measures_equality}
	Let $S \subset \RR^2$ be $2$-admissible. Then its tropical perimeter satisfies $\mathcal{M}^{1}_{\mathrm{tr}}(\partial S) = \sigma^{1}_{\mathrm{tr}}(\partial S)$.
\end{proposition}

\noindent {\bf Dimension} ${\bf n=3.}$
Let $\varv = (a,b,c)$, $c \neq 0$, be a unit vector.
We parametrize the plane $ax+by+cz=0$ by $(x,y,z)=(s,t,-\frac{a}{c}s-\frac{b}{c}t)=r(s,t)$.
The Jacobian is
$$ 
J = \,\parallel \partial_s r \times \partial_t r \parallel \, =
 \, \left\lVert \left(\tfrac{a}{c}, \tfrac{b}{c}, 1 \right) \right\rVert = \frac{\sqrt{a^2+b^2+c^2}}{\sqrt{c^2}} = \frac{1}{\vert c \vert}.
$$
Let $D$ be the central section $\varv^{\perp} \cap \Bthree$ in the $(s,t)$-coordinates. It is defined by the inequalities $\vert s \vert \leq 1$, $\vert t \vert \leq 1$, $\vert as+bt \vert \leq \vert c \vert$, $\vert s-t \vert \leq 1$, $\vert (a+c)s+bt \vert \leq \vert c \vert$ and $\vert as+(b+c)t \vert \leq \vert c \vert$. Thus, the section area is
$$
\mathcal{L}^{2}(\varv^{\perp} \cap \Bthree) = \mathcal{L}_{(s,t)}^{2}(D) \cdot J.
$$
The next examples show that in dimension 3 the two notions of surface area agree for some sets $S$ but for others $\mathcal{M}^{2}_{\mathrm{tr}}(\partial S) > \sigma^{2}_{\mathrm{tr}}(\partial S)$, depending on the direction of the normal vectors.
\begin{examples}
	\begin{enumerate}
		\item Let $F$ be a facet of a polytope supported by the plane $x+y-2z=0$, with $\varv = \frac{1}{\sqrt{6}}(1,1,-2)$ the unit normal vector.
		Then $(x,y,z) = (s,t,\frac{s+t}{2})$, $J = \sqrt{\frac{3}{2}}$,
		$D = ([-1,1] \times [-1,1]) \cap \{\lvert s-t \rvert \leq 1 \}$
		and $\mathcal{L}_{(s,t)}^{2}(D)=3$.
		So, the area of the base is $\mathcal{L}^{2}(\varv^{\perp} \cap \Bthree) = 3 \cdot \sqrt{\frac{3}{2}}$.
		The height is $h_{\Bthree}(\varv) = \langle (1,1,0), \varv \rangle = \sqrt{\frac{2}{3}}.$
		We have
		$$
		\mathcal{L}^{2}(\varv^{\perp} \cap \Bthree) \cdot h_{\Bthree}(\varv) =
		3 \cdot \sqrt{\frac{3}{2}} \cdot \sqrt{\frac{2}{3}} = 3,
		$$
		showing that in this example $\mathcal{M}^{2}_{\mathrm{tr}}(F) = \sigma^{2}_{\mathrm{tr}}(F)$.
		\item Let $F$ be on the plane $2x+2y-z=0$, $\varv = \frac{1}{3}(2,2,-1)$.
		Then $(x,y,z) = (s,t,2(s+t)$, $J = 3$
%		$D = ([-1,1] \times [-1,1]) \cap \{\lvert s-t \rvert \leq 1 \}$
		and $\mathcal{L}_{(s,t)}^{2}(D)=\frac{11}{12}$.
		So, the area of the base is $\mathcal{L}^{2}(\varv^{\perp} \cap \Bthree) = 3 \cdot \frac{11}{12}=\frac{11}{4}$.
		The height is $h_{\Bthree}(\varv) = \langle (1,1,0), \varv \rangle = \frac{4}{3}.$
		We have
		$$
			\mathcal{L}^{2}(\varv^{\perp} \cap \Bthree) \cdot h_{\Bthree}(\varv) =
			\frac{11}{4} \cdot \frac{4}{3} = \frac{11}{3} > 3,
		$$
		showing that in this example $\mathcal{M}^{2}_{\mathrm{tr}}(F) > \sigma^{2}_{\mathrm{tr}}(F)$.

	\end{enumerate}
\end{examples}

\section{Tropical honeycomb theorem in $\Rn$}
\label{sec:honecomb_Rn}
As a corollary of the tropical isoperimetric inequality we have the following.
\begin{theorem}[Tropical honeycomb theorem in $\RR^n$]
	Let $G \subset \Rn$ be an $(n-1)$-dimensional connected set dividing $\Rn$ into infinitely many bounded simply connected components of (tropical) unit measure and having an ($n-1$)-rectifiable boundary. Let $C = \Rn \setminus G$. Let $B^n_R$ be the standard $n$-dimensional ball of radius $R$ with center at the origin. Then
	\begin{equation}
		\label{eq:tr_honecomb_Rn}
		\limsup_{R \to \infty}\frac{\mathcal{M}^{n-1}_{\mathrm{tr}}(G \cap B^n_R)} {\mathcal{L}^n(B^n_R)} \geq \frac{n \sqrt[n]{n+1}}{2},
	\end{equation}
	with equality holding for components in the form of tropical unit balls.
\end{theorem}
\begin{proof}
	The proof follows from Theorem~\ref{thm:tr_isoperim_inequality}, the tropical isoperimetric inequality, and Theorem~\ref{thm:honeycomb}, similar to the proof of Theorem~\ref{thm:tr_honecomb} in dimension 2.
	By Theorem~\ref{thm:tr_isoperim_inequality}, for an $n$-dimensional body of unit (tropical) volume the minimal surface area is $n\sqrt[n]{n+1}$, which is  achieved for tropical balls. 
	Since in the tiling of $\Rn$ each facet belongs to two tropical balls and since the total surface of the tropical balls that are in the boundary of the ball $B_{R}^n$ can be ignored in the ratio when $R \to \infty$ (it is of order $O(R^{n-1})$, whereas the volume of the ball is of order $O(R^n)$), we get that the limit in \eqref{eq:tr_honecomb_Rn} for tropical balls is $\frac{n \sqrt[n]{n+1}}{2}$.  
\end{proof}
\subsection{Tiling of $\RR^3$}
	The lower bound of the honeycomb theorem in space in the tropical setting is $\frac{3}{2} \sqrt[3]{4} \approx 2.381$. In order to compare it to tiling of $\RR^3$ in the Euclidean setting, we use the data from the table in \cite{WP94} that shows the isoperimetric quotient
	\begin{equation}
		\label{eq:isoperim_quotient}
		\frac{36\pi V^2}{A^3}
	\end{equation}
	of different structures, where $V$ is the volume and $A$ is the surface area.
	This expression is used for comparison to the standard sphere, which has value 1 (a larger value refers to a more efficient tiling). Of course, a tiling of $\RR^3$ with spheres is not possible, and the best arrangement (Kepler's conjecture) was proved in 1998 by Thomas Hales (the proof was computer-assisted and a formal proof was given in 2017 \cite{Hal17}). The best known value for the isoperimetric quotient \eqref{eq:isoperim_quotient} for tiling $\RR^3$ is that of Weaire and Phelan \cite{WP94}, which is composed of two kinds of cells: an irregular dodecahedron and a truncated hexagonal trapezohedron, both being slightly curved. It is larger than the suggestion of Kelvin, more than hundred years earlier, consisting of cells in the form of a slightly curved truncated octahedron. Table~\ref{tab:isoperim_quotient} shows the values of the isoperimetric quotient for tiling with Weaire-Phelan cells, with those of Kelvin and also with cubes (taken from \cite{WP94}), in addition to $3$-dimensional tropical balls $\Bthree$ in the Euclidean setting (where the tiling is not supposed to be optimal) and in the tropical setting (the best quotient): $\frac{36\pi 4^2}{(3 \cdot 4)^3} = \frac{\pi}{3} \approx 1.047$.
\begin{center}
\begin{table}[h]
	\begin{tabular}{|c|c|} 
		\hline
		Structure & Isoperimetric quotient \\
		\hline \hline
		Standard sphere  & 1.000 \\
		\hline
		Weaire-Phelan & 0.764  \\
		\hline
		Kelvin & 0.757  \\
		\hline
		Cube & 0.524  \\
		\hline
		Tropical ball (Euclidean measure) & 0.595  \\
		\hline
		Tropical ball (tropical measure) & 1.047  \\
		\hline
	\end{tabular}	
	\vspace{5pt}
	\caption{Isoperimetric quotient $\frac{36\pi V^2}{A^3}$ for different 3D structures}
	\label{tab:isoperim_quotient}
\end{table}
\end{center}

\bibliographystyle{plain}
\bibliography{Tropical_measure}
\end{document}